\documentclass[american,10pt]{amsart}
\pdfoutput=1
\usepackage[T1]{fontenc}
\usepackage[utf8]{inputenc}
\usepackage{lmodern}
\usepackage{enumitem}
\usepackage{mathtools}

\mathtoolsset{showonlyrefs}
\usepackage{amssymb}
\usepackage{bm}
\usepackage[unicode=true,pdfusetitle,
 bookmarks=true,bookmarksnumbered=false,bookmarksopen=false,
 breaklinks=false,pdfborder={0 0 1},backref=false,colorlinks=false]{hyperref}
\usepackage[foot]{amsaddr}
\usepackage{cite}

\usepackage{subcaption}
\usepackage{pgfplots}
\pgfplotsset{compat=1.18}

\theoremstyle{plain}
\newtheorem{thm}{Theorem}[section]
\newtheorem{lem}[thm]{Lemma}
\newtheorem{prop}[thm]{Proposition}

\newtheorem{cor}[thm]{Corollary}

\theoremstyle{definition}

\theoremstyle{remark}
\newtheorem{rem}{Remark}[section]
\newtheorem*{rem*}{Remark}

\newcommand\R{\mathbb{R}}
\newcommand\N{\mathbb{N}}
\newcommand\upd{\textup{d}}
\newcommand\upe{\hspace{1pt}\textup{e}}

\usepackage{todonotes}

\usepackage{xcolor}
\newcommand{\revision}[1]{#1}
\newcommand{\secondrevision}[1]{#1}

\title[A nonlocally pulling patch of bounded size]{Spreading properties of the Fisher--KPP equation when the intrinsic growth rate is maximal in a moving patch of bounded size}

\author{Thomas Giletti}
\address[T. G.]{LMBP, Universit\'{e} Clermont-Auvergne, Campus Universitaire des C\'{e}zeaux, 3~place Vasarely, 63178 Aubi\`{e}re Cedex, France}
\email{thomas.giletti@uca.fr}

\author{L\'{e}o Girardin}
\address[L. G.]{CNRS, Institut Camille Jordan, Universit\'{e} Claude Bernard Lyon-1, 43 boulevard du 11 novembre 1918, 69622 Villeurbanne Cedex, France}
\email{leo.girardin@math.cnrs.fr}

\author{Hiroshi Matano}
\address[H. M.]{Meiji Institute for Advanced Study of Mathematical Sciences, Meiji University, Tokyo \revision{164}-8525, Japan}
\email{matano@meiji.ac.jp}

\begin{document}


%%\maketitle

\begin{abstract}
This paper is concerned with spreading properties of space-time heterogeneous Fisher--KPP equations in one space dimension. We focus on the case of everywhere favorable environment with three different zones, a left half-line with slow or intermediate growth, a central patch with fast growth and a right half-line with slow or intermediate growth. The central patch
moves at various speeds. The behavior of the front changes drastically depending on the speed of the central patch. Among other things, intriguing phenomena such as \textit{nonlocal pulling} and \textit{locking} may occur, which would make the behavior of the front further complicated. The problem we discuss here is closely related to questions in biomathematical modelling. 
By considering several special cases, we illustrate the remarkable diversity of possible behaviors. 
In particular, \revision{when the central patch has constant size and constant speed, we provide a complete set of explicit formulas for the spreading
speed.}
\end{abstract}

\maketitle

%%%%%%%%%%%%%%%%%%%
%%%%%%%%%%%%%%%%%%%
\section{Introduction}

This paper is concerned with the Cauchy problem associated with the reaction--diffusion equation
\begin{equation}\label{eq:main}\tag{KPP}
\partial_t u = \partial_{xx} u + f(t,x,u),
\end{equation}
where $t>0$ is a time variable, $x\in\R$ is a one-dimensional space variable, and $u$ is a population density function of time and space. The function~$f$ is a reaction term which is assumed to be globally bounded with respect to the variables $t$ and $x$, and of class~$\mathcal{C}^2$ with
respect to the variable $u$. Moreover, it is of the so-called KPP type with respect to $u$, namely:
\begin{equation}\label{ass:KPP}\tag{A1}
\forall(t,x)\in\R^2 , \quad
\begin{cases}
f(t,x,0)=0, \\
\partial_u f(t,x,0)>0, \\
\forall u\geq 0\quad \partial_u f(t,x,0)u\geq f(t,x,u) \geq \partial_u f(t,x,0) u - M u^2, \\
\forall u>1\quad f(t,x,u)<0,
\end{cases}
\end{equation}
where $M$ is a given positive constant.

Denoting $$r:(t,x)\mapsto \partial_u f(t,x,0),$$ we assume it has the following form:
\begin{equation}\label{ass:piecewise_constant_r}\tag{A2}
r:(t,x)\mapsto
\begin{cases}
r_1 & \text{if }x<A(t), \\
r_2 & \text{if }A(t)\leq x<A(t)+L, \\
r_3 & \text{if }A(t)+L\leq x,
\end{cases}
\end{equation}
with $r_1,r_2,r_3, L$ positive constants and $A\in\mathcal{C}\left(\R,[0,+\infty)\right)$.

The prototypical example of function $f$ we have in mind comes from ecology and has the form $f(t,x,u)=r(t,x) \left( u-u^2/K(t,x) \right)$, where $r$ is an intrinsic growth rate and $K$ a carrying capacity. In order for \eqref{ass:KPP} to be satisfied, $K$ must then satisfy that $0 < \inf K \leq \sup K \leq 1$ (notice that $0<\inf K\leq \sup K < + \infty$ would actually be enough up to some rescaling). The assumption~\eqref{ass:piecewise_constant_r} implies that there is a time-dependent zone $\Omega_0(t)$ of size $L$ where the intrinsic growth rate is $r_2$, while the intrinsic growth rate is $r_1$ (resp. $r_3$) on the left (resp. right) side of $\Omega_0(t)$. This situation is illustrated in Figure~\ref{fig:twot}, though the actual ordering of the values $r_1,r_2,r_3$ may vary throughout this paper.

\begin{figure}
    \centering
    \includegraphics[bb=0.0 0.0 1152.1 527.45, width=.89\linewidth]{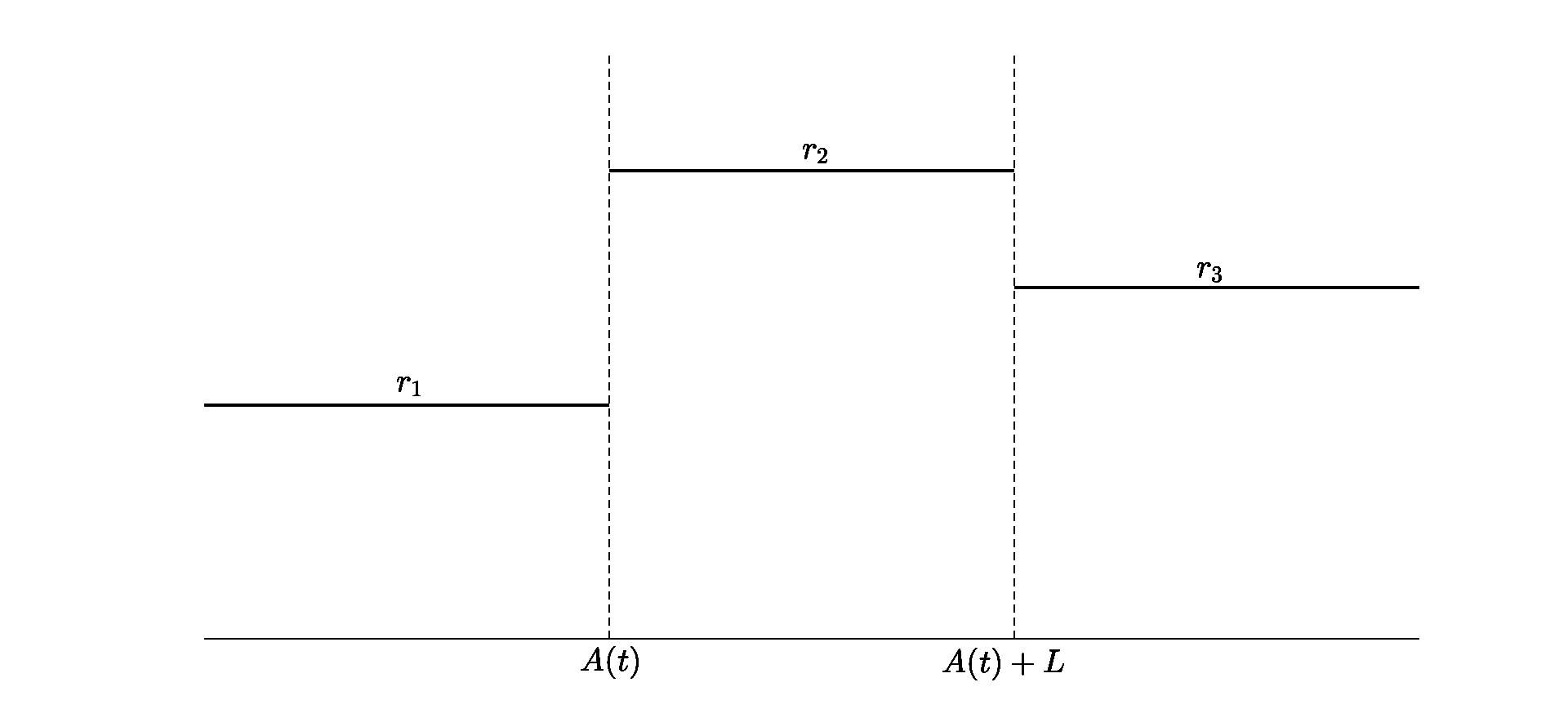}
    \caption{Illustration of the heterogeneous intrinsic growth rate under assumption~\eqref{ass:piecewise_constant_r}.}
    \label{fig:twot}
\end{figure}

The equation \eqref{eq:main} is supplemented in the Cauchy problem with initial data of the form
\begin{equation}\label{ass:initial_data}\tag{A3}
u_0\in\mathcal{C}_c(\R),\quad u_0\geq 0,\quad u_0\neq 0,
\end{equation}
where $\mathcal{C}_c (\R)$ denotes the set of compactly supported continuous functions on~$\R$. We will investigate the large time behavior of solutions, and more specifically their spreading properties. What exactly is meant here by spreading will be clarified below.

%%%%%%%%%%%%
\subsection{Some known results} 

Consider the above-mentioned prototypical example $f(t,x,u)=r(t,x)(u-u^2 /K)$ with $K=1$.  When $r_1=r_2=r_3$, the equation \eqref{eq:main} then reduces to the classical Fisher--KPP equation in homogeneous media:
\begin{equation}%\label{eq:classical_Fisher-KPP}
\partial_t u = \partial_{xx} u + ru(1-u),
\end{equation}
for which it is well known \cite{Aronson_Weinbe_1} that solutions of the Cauchy problem emanating from compactly supported initial data 
asymptotically spread at speed $2\sqrt{r}$ in the following sense:
\begin{equation}%\label{eq:classical_spreading}
\begin{array}{rcl}
2\sqrt{r}& = & \sup\left\{ c\geq 0\ |\ \lim_{t\to+\infty}\sup_{|x|\leq ct}|1-u(t,x)|=0 \right\}\\
&=& \inf\left\{ c\geq 0\ |\ \lim_{t\to+\infty}\sup_{ct\leq |x|}u(t,x)=0 \right\}.
\end{array}
\end{equation}

However, much less is known when $r_1$, $r_2$ and $r_3$ differ. Yet heterogeneous growth rates are natural to consider, for at least two reasons. First, in the last few years many studies have been devoted to
the so-called ``climate change problem'' 
%\cite{Berestycki_Die, Berestycki_Fang_2018, Bouhours_Gilet, Alfaro_Beresty}.
%Second, when considering more complicated population dynamics models that
%take the form of reaction--diffusion systems, it is often possible to characterize the speed of the fastest species, but the speeds
%of the second, third, etc., species generally remain elusive: indeed, how does
%the environmental change, at the interface generated by the spreading front of the first species, affect the speeds of the slower species? 
%Are the slower species able to take profit of the environment ahead of the first front despite the increasing distance? 
\cite{Berestycki_Die, Berestycki_Fang_2018, Bouhours_Gilet, Alfaro_Beresty}, which deals with the situation where the favorable zone moves at a certain speed. Our second motivation comes from the analysis of multi-species models. More precisely, when one considers population models that take the form of reaction-diffusion systems involving multiple species, it is often possible to characterize the spreading speed of the fastest species, but the speeds of the second, third and other slower species generally remain elusive. This is because, although the position of the fronts of slower species lie much behind that of the fastest species, their thin leading edge stretches far beyond the front of the fastest species; therefore it is not immediately clear how the environmental change caused by the spreading front of the fastest species affects the speeds of the slower species. Are the slower species fully adapted to the new environment generated by the front of the fastest species? Or can they somewhat benefit from the environment ahead of the first front despite the increasing distance? 
Quantitative sharp answers are known only in the simplest cases 
\cite{Carrere_2017,Ducrot_Giletti_Matano_2,Girardin_Lam,Holzer_Scheel,Bovier_Hartung_2022} 
and their proofs typically rely on a very specific structure in the system (directions of instability, comparison principle, 
decoupled equation, etc.). 
A better understanding of the heterogeneous scalar equation \eqref{eq:main} should help to obtain more general results. See Section~\ref{sec:multi-species} below where the relation between the equation \eqref{eq:main} for the case $r_2>\max(r_1,r_3)$ and some three-species competition or two-prey-one-predator population models is explained.

Now let us recall some known results on the heterogeneous scalar equation \eqref{eq:main}.  When the heterogeneity remains confined ($A(t)$ being globally bounded), only the asymptotic space-time growth rate matters and determines the spreading speed \cite{Berestycki_Nadin_2012}, which in our case means that the population spreads toward the left at speed $2\sqrt{r_1}$ and toward the right at speed $2\sqrt{r_3}$.

\begin{figure}
    \centering
    \includegraphics[bb=0.0 0.0 1439.82 638.92, width=.89\linewidth]{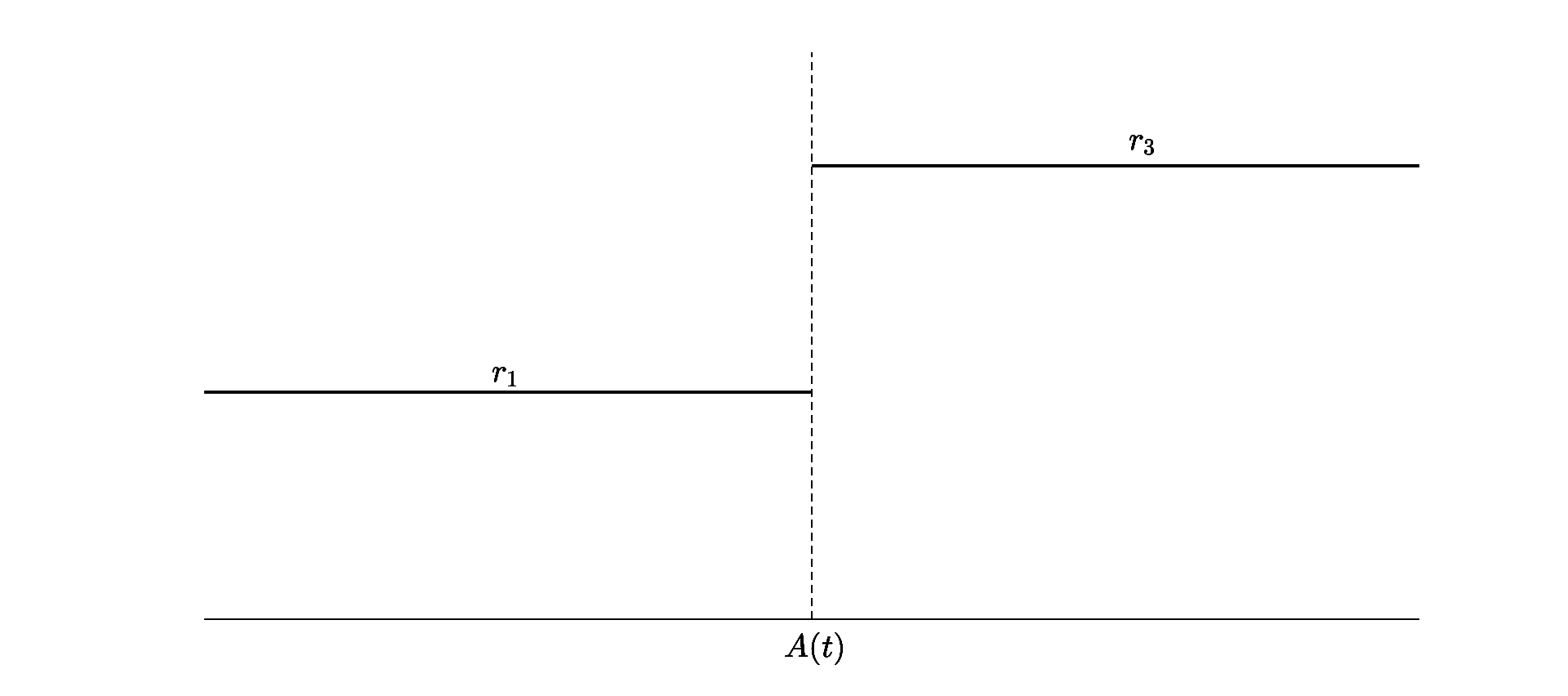}
    \caption{Illustration of the heterogeneous intrinsic growth rate under the assumptions of Theorem~\ref{thm:one_interface}.}
    \label{fig:onet}
\end{figure}

Next, in the special case where $r_1 \neq r_2 = r_3$, only a single transition appears in the intrinsic growth rate; see Figure~\ref{fig:onet}. \revision{If} $A(t)$ moves to the right too fast, then the spreading front is located far to the left of $A(t)$, where the intrinsic growth rate equals $r_1$. Thus \revision{one may speculate that} the spreading speed is likely to be $2\sqrt{r_1}$, as in the homogeneous KPP equation. On the other hand, if $A(t)$ moves too slowly, then the spreading front is located far ahead of $A(t)$, where $r = r_3$. Thus the spreading speed is likely to be \revision{$2\sqrt{r_3} $}. 
These speculations turn out to be correct \revision{in these cases}, but what if the speed of $A(t)$ is neither too fast nor too slow? 

Actually the answer is not so simple \revision{if the speed of $A(t)$ is in the intermediate range: indeed,} totally different types of phenomena, \textit{locking} and \textit{nonlocal pulling}, are 
\revision{known to occur}, depending on whether $r_1 > r_3$ or $r_1 < r_3$. 
\revision{Here, by ``locking'', we mean a situation where the spreading front is trapped around the position of $A(t)$ and travels at the same speed as $A(t)$. This can possibly happen if $r_1>r_3$ and the speed of $A(t)$ is between $2\sqrt{r_1}$ and $2\sqrt{r_3}$. By ``nonlocal pulling'', we mean a situation where $A(t)$ travels faster than the front and lies far ahead of it, yet the larger intrinsic growth rate $r_3$ in the zone $x>A(t)$ exerts a substantial influence on the speed of the front despite the growing distance. This can possibly happen if $r_1<r_3$ and the speed of $A(t)$ is larger than $2\sqrt{r_3}$ but not too large.} 
These two terms, \textit{locking} and \textit{nonlocal pulling},  were first introduced in \cite{Holzer_Scheel} and \cite{Girardin_Lam},  respectively\footnote{In \cite{Holzer_Scheel}, nonlocally pulled fronts were referred to as accelerated fronts. But, in the Fisher--KPP literature, acceleration also refers to superlinear spreading. We believe that nonlocal pulling better highlights the underlying mechanism, which is that a small population in a moving frame strictly faster than the actual front may still contribute to the propagation due to variations in the intrinsic growth rate.}.

In the simplest case when the transition moves with constant speed, \textit{i.e.}, $A(t)$ is a linear function, the known results (that were proved in possibly more complex contexts
in \cite{Girardin_Lam,Holzer_Scheel,Bovier_Hartung_2022,Lam_Yu_2021}) 
can be summarized as follows.

\begin{thm}[A single transition~\cite{Lam_Yu_2021}]\label{thm:one_interface}
Assume $r_1\neq r_2=r_3$ and the existence of $c_A\geq 0$ such that, for all $t\geq 0$, 
$$ A(t)=c_A t.$$
Then the asymptotic spreading speed $c^\star>0$ of $u$ is well-defined as 
\begin{equation}
c^\star=\sup\left\{ c\geq 0\ |\ \liminf_{t\to+\infty}\inf_{0\leq x\leq ct} u(t,x)>0 \right\}=\inf\left\{ c\geq 0\ |\ \lim_{t\to+\infty}\sup_{ct\leq x}u(t,x)=0 \right\}
\end{equation}
and satisfies:
\begin{enumerate}[label=(\alph*)]
\item if $r_1>r_3$:
\begin{enumerate}[label=(\roman*)]
\item $c^\star=2\sqrt{r_3}$ \ if \ $c_A\leq 2\sqrt{r_3}$,
\item $c^\star=c_A$ \ if \ $2\sqrt{r_3}<c_A\leq 2\sqrt{r_1}$ (locked front),
\item $c^\star=2\sqrt{r_1}$ \ if \ $2\sqrt{r_1}<c_A$;
\end{enumerate}
\item if $r_1<r_3$:
\begin{enumerate}[label=(\roman*)]
\item $c^\star=2\sqrt{r_3}$ \ if \ $c_A\leq 2\sqrt{r_3}$,
\item $c^\star=\frac{c_A-2\sqrt{r_3-r_1}}{2}+\frac{2r_1}{c_A-2\sqrt{r_3-r_1}}$ \ if \ $2\sqrt{r_3}<c_A\leq 2\sqrt{r_1}+2\sqrt{r_3-r_1}$ (nonlocally pulled front),
\item $c^\star=2\sqrt{r_1}$ \ if \ $2\sqrt{r_1}+2\sqrt{r_3-r_1}<c_A$.
\end{enumerate}
\end{enumerate}
\end{thm}
Since these six cases form the basis for our intuition, let us describe briefly what happens in
each case. The statement is also summarized in Figure \ref{fig:no_patch_with_constant_speed}.

Cases $(a)(i)$ and $(b)(i)$ correspond to a situation outlined above where the transition $A(t)$ moves so slow that the population moves ahead of it at speed $2\sqrt{r_3}$. In cases $(a)(iii)$ and $(b)(iii)$, the transition is so fast that the spreading speed $2 \sqrt{r_1}$ is determined only by the left side zone. 
\revision{More intriguing behaviors are observed in the intermediate cases $(a)(ii)$ and $(b)(ii)$. Case $(a)(ii)$} is the \textit{locking case}, where the population spreads exactly at the \revision{same speed $c_A$ as the transition point $A(t)$}.
%spreading population density catches up with the transition is the intermediate case for decreasing growth rates: on one hand, exponentially 
%small populations ahead of the interface are unable to grow, but on the other hand, the KPP 
%speed in the wake of the interface is larger than the speed of the interface itself. The 
%spreading population density catches up with the interface but is unable to outdistance it.
%This is the \textit{locking case}.
Finally, case $(b)(ii)$ is the \textit{nonlocal pulling case}, 
\revision{where $A(t)$ travels faster than the solution front, therefore the intrinsic growth rate is $r_1$ in a vast area around the front, yet the speed $c^\star$ of the front is strictly larger than $2\sqrt{r_1}$. This implies that the larger intrinsic growth rate $r_3$ in the zone $x>A(t)$ exerts a substantial influence on the speed of the front despite the growing distance between $A(t)$ and the front. 
%%This is remarkable considering that the leading edge of the solution is extremely small in the zone $x>A(t)$ and decays exponentially in time.
}
%We point out that the thresholds separating on one hand cases $(i)(a)$, %$(i)(b)$, $(i)(c)$, and on the other hand
%cases $(ii)(a)$, $(ii)(b)$, $(ii)(c)$ differ. 

\begin{figure}
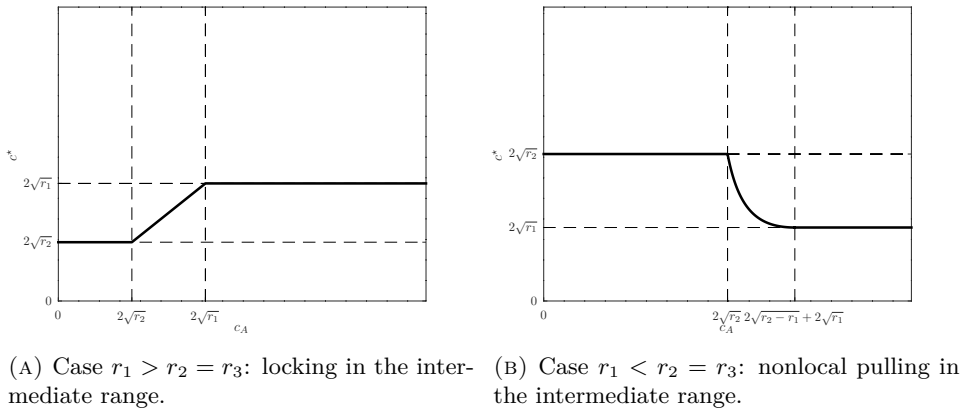

    \begin{subfigure}{.49\linewidth}
        \resizebox{\linewidth}{!}{\input{speed_r1=4_r2=1.tex}}
        \caption{Case $r_1>r_2 = r_3$: locking in the intermediate range.}
    \end{subfigure}
    \hfill
    \begin{subfigure}{.49\linewidth}
        \resizebox{\linewidth}{!}{\input{speed_r1=1_r2=4.tex}}
        \caption{Case $r_1<r_2 =r_3$: nonlocal pulling in the intermediate range.}
    \end{subfigure}
    
    \caption{The spreading speed as a function of the environmental speed under
    the assumptions of Theorem \ref{thm:one_interface}.}
    \label{fig:no_patch_with_constant_speed}
\end{figure}

Note that the persistence behind the propagating front is described as 
\begin{equation}
    \liminf_{t \to +\infty}\inf_{0\leq x\leq ct}u(t,x)>0.
\end{equation}
This generalizes appropriately the classical picture
\begin{equation}
    \lim_{t\to+\infty}\sup_{|x|\leq ct}|1-u(t,x)|=0
\end{equation}
to our more general setting where the local behavior of the solution is typically not a convergence to a constant
steady state, which may not exist here, but rather complicated oscillations.

The more general case (more than one unconfined interface) remained elusive until very recently. In 2021, Lam and Yu 
\cite{Lam_Yu_2021} managed to characterize the spreading speed in many cases thanks to a Hamilton--Jacobi framework. 
Although the formula is in general complicated and implicit, in some cases it can be made explicit. In particular, 
in the case of two interfaces $A(t)$ and $A(t)+L$ with constant successive values 
$r_1$, $r_2$ and $r_3$, all subcases can be solved provided the central patch
does not have a positive effect on the growth rate from both sides (\textit{i.e.}, provided $r_2<\max(r_1,r_3)$).
This assumption is not a technical issue but rather a true limitation of the Hamilton--Jacobi approach, which involves a hyperbolic rescaling. Basically, this approach cannot deal with nonlocal pulling exerted from areas that are too small. This is why we will use a different approach here, based on super- and sub-solutions, to study this more delicate type of nonlocal pulling
\footnote{\revision{After the submission of our manuscript, Lam, Nadin and Yu published a paper \cite{Lam_Yu_Nadin_2025}, which refines the Hamilton--Jacobi approach with flux-limited solutions, thereby making it possible to deal with the case $r_2>\max(r_1,r_3)$. Their results are consistent with ours.}}

To conclude this general introduction, we point out a very recent work by Bovier and Hartung \cite{Bovier_Hartung_2022} where 
a probabilistic method was used to quantify rigorously the nonlocal pulling in a model for the spread of a trait in a population, originally introduced in \cite{Venegas_Ortiz_}. 
It is not clear to us whether this third method of proof could be used to deal with the case we consider here.

%%%%%%%%%%%%%%%
\subsection{Organization of the paper}
In Section 2 we present our main results, which \revision{focus on} the \revision{case} when the intrinsic growth rate is larger in the intermediate zone between the two transitions.

The proofs of our results are \revision{given} in Section~3, using \revision{an elaborate} construction of sub- and super-solutions \revision{involving the principal eigenvalue of} an eigenvalue problem in the \revision{moving frame}. In Section 4 we \revision{summarize} a few interesting properties \revision{of this principal} eigenvalue and discuss possible extensions of our work.

%%%%%%%%%%%%%%%%%%%
%%%%%%%%%%%%%%%%%%%
\section{Main results}\label{s:main}

We now consider the equation~\eqref{eq:main}, supplemented with the assumptions~\eqref{ass:KPP}, \eqref{ass:piecewise_constant_r}, \eqref{ass:initial_data}. Throughout the rest of this paper, we assume that
\begin{equation}\label{r2>r1r3}
r_2>\max(r_1,r_3).
\end{equation}
This means that the intrinsic growth rate is at its highest in the central patch of positive and bounded length between the two moving transitions. As we explained in the previous section, such a situation has so far remained unaddressed in the mathematical literature, partly due (among other reasons) to the difficulty of estimating the aforementioned nonlocal pulling effects for this case.

From now on, we only consider the rightward spreading properties. 
Since $0\leq A(t)\leq A(t)+L$, it is indeed standard to prove that the leftward
spreading speed is well-defined and equals $2\sqrt{r_1}$.

Before stating the results, we define two convenient quantities, the \textit{minimal spreading speed} $\underline{c}$
and the \textit{maximal spreading speed} $\overline{c}$:
\begin{equation}
\underline{c}=\sup\left\{ c\geq 0\ |\ \liminf_{t\to+\infty}\inf_{0\leq x\leq ct} u(t,x)>0 \right\},
\end{equation}
\begin{equation}
\overline{c}=\inf\left\{ c\geq 0\ |\ \lim_{t\to+\infty}\sup_{ct\leq x}u(t,x)=0 \right\}.
\end{equation}
These two quantities are well-defined since the initial value $u_0$ satisfies \eqref{ass:initial_data}. By comparison with solutions of homogeneous problems, one may verify that
\begin{equation}
    2\sqrt{\min(r_1,r_3)}\leq\underline{c}\leq\overline{c}\leq 2\sqrt{r_2}.
\end{equation}
The equality between $\underline{c}$ and $\overline{c}$ is a very natural question. For more general initial data $u_0$,
it is known that the equality can fail even if $r_1=r_2=r_3$ \cite{Hamel_Nadin_2012}. When there is indeed equality between $\underline{c}$ 
and $\overline{c}$, the quantity $\underline{c}=\overline{c}$ is referred to as the \textit{spreading speed}; 
in other words, the spreading speed is well-defined if, and only if, $\underline{c}=\overline{c}$.
Our results will in particular assert that, when $r_2>\max(r_1,r_3)$ and~$u_0$ is compactly supported, the
spreading speed can be well-defined or ill-defined, depending on the variations of $A$.
In fact it will become clear that the range of possible outcomes is very wide, \revision{which makes it difficult to capture the entire picture.} 
Therefore, \revision{in the present paper we will focus on solving} extreme cases and will illustrate the \revision{remarkable} variety of outcomes.

We begin with the case where $A$ is a linear function of time. This is our main result.

\begin{thm}[Patch of constant size, constant speed]\label{thm:constant_case}
Assume the existence of $c_A >0$ such that, for all $t\geq 0$, 
$$A(t)=c_At . $$
Let 
\begin{equation}
\underline{L}=
\begin{cases}
0 & \text{if }r_1=r_3, \\
\frac{1}{\sqrt{r_2-\max(r_1,r_3)}}\operatorname{arccot}\left(\sqrt{\frac{r_2-\max(r_1,r_3)}{|r_1-r_3|}}\right) & \text{if }r_1\neq r_3,
\end{cases}
\end{equation}
where $\operatorname{arccot}$ denotes the inverse of $\cot_{|(0,\pi)}$. Then:
\begin{equation}\label{eq:speed_formula}
\underline{c}=\overline{c}=
\begin{cases}
2\sqrt{r_3} & \text{if }c_A<2\sqrt{r_3}, \\
c_A & \text{if }2\sqrt{r_3}\leq c_A\leq 2\sqrt{-\lambda_1}\quad\text{(locked front)}, \\
F(c_A) & \text{if }2\sqrt{-\lambda_1}<c_A<2\sqrt{r_1}+2\sqrt{-\lambda_1-r_1}\\
& \quad\text{(nonlocally pulled front)}, \\
2\sqrt{r_1} & \text{if } 2\sqrt{r_1}+2\sqrt{-\lambda_1-r_1}\leq c_A,
\end{cases}
\end{equation}
where
\begin{itemize}
\item $\lambda_1 = - \max (r_1,r_3)$ if $L \leq \underline{L}$;
%\item $\lambda_1=-r_1$ if $r_1>r_3$ and $L\leq\underline{L}$;
%\item $\lambda_1=-r_3$ if $r_1<r_3$ and $L\leq\underline{L}$;
\item $\lambda_1\in(-r_2,-\max(r_1,r_3))$ if $L>\underline{L}$, and in this case $\lambda_1$
is characterized as the unique solution in $\left(-r_2,\min\left(-\max\left(r_1,r_3\right),\frac{\pi^2}{L^2}-r_2\right)\right)$ of:
\begin{equation}\label{eq:equation_defining_lambda1}
\cot(L\sqrt{r_2+\lambda_1})=\frac{r_2+\lambda_1-\sqrt{(r_1+\lambda_1)(r_3+\lambda_1)}}{\sqrt{r_2+\lambda_1}(\sqrt{-r_1-\lambda_1}+\sqrt{-r_3-\lambda_1})};
\end{equation}
\item the function $F$ is defined by:
\begin{equation}\label{eq:def_F}
    F:c\in\left(2\sqrt{-\lambda_1-r_1},+\infty\right)\mapsto \frac{c-2\sqrt{-\lambda_1-r_1}}{2}+\frac{2r_1}{c-2\sqrt{-\lambda_1-r_1}};
\end{equation}
\end{itemize}
\end{thm}

The intervals $(2\sqrt{r_3},2\sqrt{-\lambda_1})$ and $(2\sqrt{-\lambda_1},2\sqrt{r_1}+2\sqrt{-\lambda_1-r_1})$ correspond to a locking situation and a nonlocally pulled situation, respectively. As is easily seen, the former interval is empty if and only if $L\leq \underline{L}$ and $r_3 > r_1$, while the latter is empty if and only if $L\leq \underline{L}$ and $r_3 < r_1$. Therefore the two intervals cannot be both empty. 
There are five possible cases, which are illustrated on Figure \ref{fig:constant_patch_with_constant_speed}.

\begin{figure}
    \begin{subfigure}{.49\linewidth}
        \resizebox{\linewidth}{!}{\input{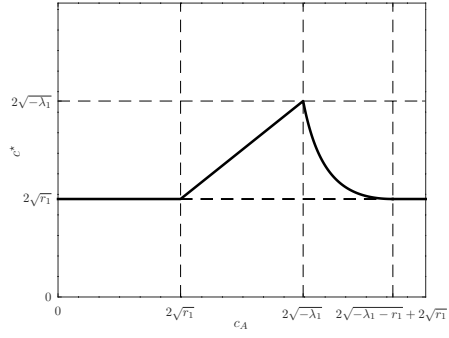}}
        \caption{Case $r_1=r_3$, $L>\underline{L}=0$ (parameter values: $r_1=1$, $r_2=9$, $r_3=1$, $\lambda_1=-4$)}
    \end{subfigure}
    
    \begin{subfigure}{.49\linewidth}
        \resizebox{\linewidth}{!}{\input{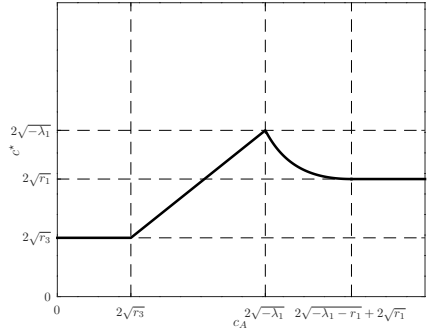}}
        \caption{Case $r_1>r_3$, $L>\underline{L}$ (parameter values: $r_1=4$, $r_2=9$, $r_3=1$, $\lambda_1=-8$)}
    \end{subfigure}
    \hfill
    \begin{subfigure}{.49\linewidth}
        \resizebox{\linewidth}{!}{\input{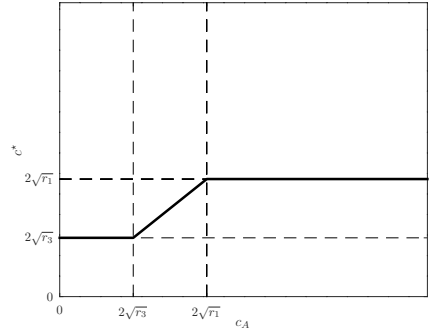}}
        \caption{Case $r_1>r_3$, $L\leq\underline{L}$ (parameter values: $r_1=4$, $r_2=9$, $r_3=1$, $\lambda_1=-4$)}
    \end{subfigure}
    
    \begin{subfigure}{.49\linewidth}
        \resizebox{\linewidth}{!}{\input{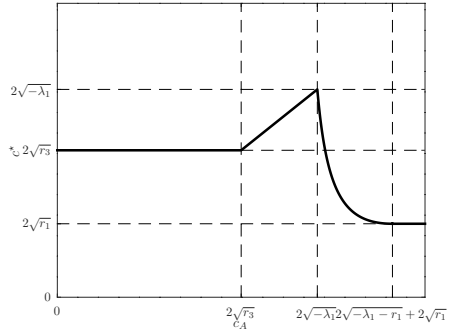}}
        \caption{Case $r_1<r_3$, $L>\underline{L}$ (parameter values: $r_1=1$, $r_2=9$, $r_3=4$, $\lambda_1=-8$)}
    \end{subfigure}
    \hfill
    \begin{subfigure}{.49\linewidth}
        \resizebox{\linewidth}{!}{\input{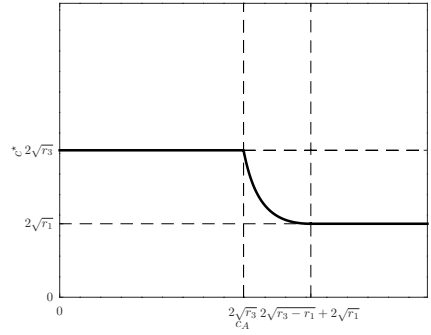}}
        \caption{Case $r_1<r_3$, $L\leq\underline{L}$ (parameter values: $r_1=1$, $r_2=9$, $r_3=4$, $\lambda_1=-4$)}
    \end{subfigure}
    
    \caption{The spreading speed $c^\star$ of Theorem \ref{thm:constant_case} as a function of the speed of the environmental 
    heterogeneity $c_A$. Note that we fix values of $\lambda_1$ instead of fixing values of $L$. 
    This is rigorously equivalent, \textit{cf.} Proposition \ref{prop:lambda1_L0}. Note also that the figures where
    $L\leq\underline{L}$, \textit{i.e.}, $\lambda_1=-\max(r_1,r_3)$, are independent of $L$ and $\lambda_1$.}
    \label{fig:constant_patch_with_constant_speed}
\end{figure}

We remark in particular that if $r_1=r_3$, then $\underline{L}=0<L$, therefore the above two intervals are both non-empty. Consequently, the patch can induce locking or nonlocal pulling, depending on the value of the
speed $c_A$, even if the patch size~$L$ is arbitrarily small. This shows that even a small perturbation of the homogeneous KPP equation may have a substantial impact on the large time behavior of solutions by substantially altering the spreading speed.

As shown in Lemma \ref{lem:eigen}, the quantity $\lambda_1$ \revision{is} the generalized principal eigenvalue of the problem 
\eqref{eq:eigenproblem1}. Its properties,
as a function of the parameters $L$ or $r_2$, will be studied in Section \ref{sec:properties_lambda1}. It is directly related to the leading eigenvalue $\lambda$ in \cite{Holzer_Scheel}, as illustrated by the correspondence between Figure \ref{fig:constant_patch_with_constant_speed_lambda1} and \cite[Figure 3-left]{Holzer_Scheel}. In particular, one may replace~\eqref{ass:piecewise_constant_r} by the more general assumption that $x\mapsto r (t,x- A(t))$ is a compact perturbation of the Heaviside type function 
\begin{equation}
\phi : x \mapsto
\begin{cases}
r_1 & \text{if }x<0, \\
r_3 & \text{if } 0 \leq x.
\end{cases}
\end{equation}
Then, when $A(t) = c_A t$, one may redefine the principal eigenvalue~$\lambda_1$ appropriately in this context and recover the formula~\eqref{eq:speed_formula} for the spreading speed. We refer to Section~\ref{sec:moregeneral} for more details.
\begin{figure}
    \resizebox{.8\linewidth}{!}{\input{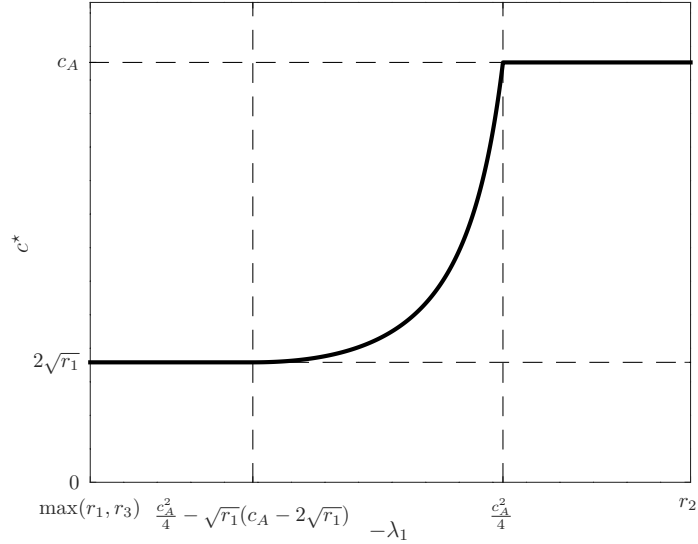}}
    \caption{The spreading speed $c^\star$ of Theorem \ref{thm:constant_case} as a function of the generalized principal eigenvalue $\lambda_1$ in the case $c_A>2\sqrt{\max(r_1,r_3)}$ (parameter values: $r_1=1$, $r_2=16$, $r_3=4$, $c_A=7$). Due to the monotonic dependence of $\lambda_1$ on $r_2$ or $L$ (the larger the patch, the larger $-\lambda_1$; 
    \textit{cf.} Propositions~\ref{prop:lambda1_L0} and \ref{prop:lambda1_r2}), this graph is a proxy for $c^\star$ as a function of $L$ or $r_2$. 
    It can be interpreted as follows: when the environmental speed $c_A$ is larger than $2\sqrt{\max(r_1,r_3)}$, 
    a small patch has no effect, an intermediate patch induces nonlocal pulling, a large patch induces locking. 
    Similar figures in the cases $2\sqrt{r_3}\leq c_A\leq 2\sqrt{r_1}$ and $c_A<2\sqrt{r_3}$ could be produced
    and would show that the spreading speed in these cases is never impacted by the patch.}
    \label{fig:constant_patch_with_constant_speed_lambda1}
\end{figure}

Regarding regularity issues, we note that the spreading speed $c^\star$ in
Theorem \ref{thm:constant_case} is
only Lipschitz-continuous as a function of $c_A$ and $\lambda_1$ (see
also Figures~\ref{fig:constant_patch_with_constant_speed} and~\ref{fig:constant_patch_with_constant_speed_lambda1}).

We continue with two corollaries confirming that if the patch $[A(t),A(t)+L]$ moves either too slowly or too \revision{quickly}, with
explicit thresholds given by the preceding theorem (accounting for $\lim_{L\to+\infty}\lambda_1(L)=-r_2$), 
then it has no effect and the front is \revision{``locally pulled''} even in the presence of arbitrary oscillations. 
\revision{Here, by ``locally pulled'', we mean that the front is pulled in the standard classical sense as in 
\cite{Holzer_Scheel,Garnier_Giletti_Hamel_Roques}, that is, the speed of the front is determined by linearization at the leading edge that lies directly ahead of the front (\textit{linear determinacy}); this is in contrast with the nonlocal linear determinacy of the speed in the nonlocal pulling case, where the environment far ahead of the front influences significantly its speed.} 

\begin{cor}[Slow patch]\label{cor:slow_patch}
Assume 
\begin{equation}
\sup_{t\geq 0}\frac{A(t)+L}{t}\leq 2\sqrt{r_3}. 
\end{equation}
Then the front is locally pulled: $\underline{c}=\overline{c}=2\sqrt{r_3}$.
\end{cor}

\begin{cor}[Fast patch]\label{cor:fast_patch}
Assume 
\begin{equation}
2\sqrt{r_1}+2\sqrt{r_2-r_1}\leq\inf_{t\geq 0}\frac{A(t)}{t}. 
\end{equation}
Then the front is locally pulled: $\underline{c}=\overline{c}=2\sqrt{r_1}$. 
\end{cor}

These two corollaries remain true if the constant $L$ is replaced by some (positive and continuous) function of time $L(t)$. We also expect that $\sup$ and $\inf$ may respectively be replaced by $\limsup$ and $\liminf$. 
For the sake of brevity, we leave the latter generalization as an open problem.

Finally, we turn to the interesting case where the patch is neither too slow nor too fast but oscillates slowly. 
\revision{For brevity, we do not seek a sharp general result here; instead,} we will construct a telling counterexample 
\revision{to the equality $\underline{c}=\overline{c}$, namely an example in which a unique spreading speed is not well defined.}

\begin{thm}[Slowly oscillating speeds]\label{thm:intermediate_case_large_oscillations}
Let $\lambda_1$ be as in Theorem \ref{thm:constant_case}, and assume that $\lambda_1\neq -r_1$. 
Let $c_{A,1}, c_{A,2}$ be positive numbers satisfying  $2\sqrt{-\lambda_1}<c_{A,1}<c_{A,2}<2\sqrt{r_1}+2\sqrt{-\lambda_1-r_1}$
and let $(t_n)_{n\in\N}\subset \R$ be an increasing sequence such that:
\[
t_0=0,\quad t_{n+1}/t_n\to+\infty\ (n\to +\infty).
\]
Let $A(t)$ be a function defined iteratively as follows:
\begin{equation}\label{A(t)}
A(t)=
\begin{cases}
\,0 & t=t_0=0,\\
A(t_{2n})+c_{A,1}(t-t_{2n}) & t\in[t_{2n},t_{2n+1})\ \ (n=0,1,2,\ldots),\\
A(t_{2n+1})+c_{A,2}(t-t_{2n+1}) & t\in[t_{2n+1},t_{2n+2})\ \ (n=0,1,2,\ldots).
\end{cases}
\end{equation}
Then:
\begin{equation}
    \underline{c}\leq\frac12\left(c_{A,2}-2\sqrt{-\lambda_1-r_1}+\frac{4r_1}{c_{A,2}-2\sqrt{-\lambda_1-r_1}}\right),
\end{equation}
\begin{equation}
    \overline{c}\geq\frac12\left(c_{A,1}-2\sqrt{-\lambda_1-r_1}+\frac{4r_1}{c_{A,1}-2\sqrt{-\lambda_1-r_1}}\right).
\end{equation}
In particular, $\underline{c}<\overline{c}$.
\end{thm}

\begin{rem}\label{rem:lambda1-r1}
The condition $\lambda_1\ne -r_1$ is satisfied if and only if either $r_1< r_3$, or $r_1 \geq r_3$ and $L >\underline{L}$.
\end{rem}

In the above example, the speed of the patch oscillates slower and slower between two values either of which would lead to nonlocal pulling in the absence of oscillations. The speed of the front then slowly fluctuates forever, thus making the spreading speed ill-defined. 
This result is inspired by \cite{Garnier_Giletti_Nadin_2012}, where temporally constant environments with slower and slower spatial oscillations lead similarly to a gap between the minimal and the maximal spreading speeds. 
It will be clear from the proof that the result can be extended straightforwardly to the case $c_{A,1}>c_{A,2}$ or to the case where $L$ or $r_2$ oscillates instead of $c_A$. 
The general conclusion we draw from this result is that oscillations of the patch, in height, width or speed, can be enough to break the uniqueness of the spreading speed, even if $r_1=r_3$. Furthermore, this result holds for an arbitrarily small~$L$, which implies that a small perturbation of the homogeneous KPP equation may not only change the spreading speed (as we already outlined above) but even make the speed ill-defined.

%%%%%%%%%%%%%%%
\subsection{Relation with multi-species models}\label{sec:multi-species}

We end this section with a discussion related to the study of spreading properties of multi-species reaction--diffusion systems, which is our main motivation 
for starting the present work. The results of the present paper show how a traveling patch of a highly favorable zone 
can strongly impact the speed of the invasion fronts, even if they travel far behind the patch. 
In particular, if a nonlocal pulling occurs, finding the precise front speeds would become a much more intricate task.   
This suggests that, when one studies the behavior of invading fronts in multi-species reaction--diffusion models, determining the speeds 
of the slower fronts could be much harder than determining those of the faster fronts, since the environmental changes created by the 
faster fronts may induce such effects as nonlocal pulling or locking on the slower fronts. 
For example, consider the following (conceptual) reaction--diffusion population models:
\begin{itemize}
\item[(A)] 3-species competition model;
\item[(B)] 2-prey \& 1-predator model.
\end{itemize}

First, in the model (A), let $u,v, w$ denote the populations of mutually competing species that invade a region at the speeds $c_u>c_v>c_w$. More precisely, 
\begin{enumerate}
\item $u$ invades the open space at speed $c_u$ as if it was alone in the environment; 
\item then the slower species $v$ replaces $u$ by competition at speed $c_v$; 
\item finally, the slowest species $w$ propagates at speed $c_w$ by replacing $v$. 
\end{enumerate}
Let us focus on the speed of $w$. 
If the interspecific competition between $u$ and $v$ is very strong, then at the interface between the two species 
(which travels at speed $c_v$), 
there appears a zone where $u, v$ are both very small. This zone can be seen as a favorable patch for the third species $w$. Therefore, although this patch lies far ahead of the front of $w$, it may still have a strong effect on the speed of $w$, just like the central patch in our equation~\eqref{eq:main}. 
At the moment, the precise effect of this patch is not fully understood, but at least it is known, though in a different context, that a strong competition between $u$ and $v$ can create a favorable patch for $w$; see \cite{Ei_Ikeda_Ogawa_2023}.

On the other hand, if the interspecific competition between $u$ and $v$ is not very strong, the superposition of the two populations would have higher density around the interface between $u$ and $v$, thus creating an unfavorable zone for $w$, contrary to the above case. The paper of Lam, Liu and Liu 
\cite[Proposition A.4]{Lam_Liu_Liu_2019} precisely identifies a parameter regime where this would indeed occur. In such a situation, the interface between $u$ and $v$ may not have much effect on the speed of $w$.

Next we consider the model (B). Let $u, v$ denote the prey populations and $w$ the predator that invade a region at speeds $c_u>c_v>c_w$, basically in the same manner as (1), (2), (3) above, except that the slowest species $w$ is a predator. We assume that the interspecific competition between $u$ and $v$ is relatively mild, so that the superposition of the two populations have higher density around the interface between $u$ and $v$, as in \cite[Proposition A.4]{Lam_Liu_Liu_2019}. Then this zone may serve as a favorable patch for the predator $w$, therefore it may have a non-negligible effect on the speed of $w$, just as in the case (A) with strong inter-specific competition between $u$ an $v$.  
In such parameter regimes, the study of the speed of the predator would require a sharp quantification of the height, width and speed of the corresponding patch, which is currently out of reach in general. The same discussion may be held in the case of one prey and two predators, for which we refer to~\cite{Ducrot_Giletti_Guo_Shimojo_2021} where among other things the nonlocal pulling phenomenon was also observed.

%%%

%%%

%%%%%%%%%%%%%%%%%%%%
%%%%%%%%%%%%%%%%%%%%
\section{Proofs}\label{s:proofs}

%%%%%%%%%%%%%%%%%
\subsection{Technical preliminaries}\label{sec:construc_eigenfct}

\subsubsection{A change of variables and a parabolic eigenproblem in the whole space-time}

In what follows we assume in addition that $A$ is of class $\mathcal{C}^2$. %and that~$L$ is pointwise positive.

Following Allwright \cite{Allwright_2021}, we can transform the linearized equation
\begin{equation}
\partial_t u - \partial_{xx} u = r(t,x)u
\end{equation}
into
\begin{equation}
\partial_t v - \frac{1}{L^2}\partial_{yy}v = \left( m(y)+\frac{y A''(t)L}{2} \right)v,
\end{equation}
where
\begin{equation}\label{m}
m(y)=r_1 \mathbf{1}_{y<0} + r_2 \mathbf{1}_{0\leq y <1} + r_3 \mathbf{1}_{1\leq y},
\end{equation}
by means of the change of variable
\begin{equation}
v(t,y)=u\left(t,Ly+A\left(t\right)\right)\exp\left(\int_0^t\frac{A'(\tau)^2}{4}\upd \tau +\frac{yA'(t)L}{2}\right).
\end{equation}

Similarly, considering an eigenproblem where $r(t,x)$ is replaced by $r(t,x)+\lambda$
changes in the new equation $m(y)$ into $m(y)+\lambda$. More precisely, inspired by 
\cite{Berestycki_Nadin_2022}, we are led to considering the following eigenproblem:
\begin{equation}
\begin{cases}
\mathcal{P}\varphi_1=\lambda\varphi_1 & \text{in }\R\times\R, \\
\varphi_1>0 & \text{in }\R\times\R, \\
\varphi_1\in \left\{\psi\in\mathcal{D}'(\R^2)\ |\ \partial_t\psi, \partial_y\psi, \partial_{yy}\psi\in L^1_{\text{loc}}(\R^2)\right\}, & 
\end{cases}
\end{equation}
with
\begin{equation}
\mathcal{P}=\partial_t - \frac{1}{L^2}\partial_{yy} - \left( m(y)+\frac{y A''(t)L}{2} \right).
\end{equation}

Let us emphasize that, to the best of our knowledge, there is no known general theory
for such eigenproblems. In particular, if $A''(t)\not\equiv 0$, keeping $\varphi, \psi$ within a bounded range is already an issue difficult to overcome. This is why from now on we focus on a more feasible special case. Other special cases where this eigenproblem makes sense might be investigated in future sequels.

%%%%%%%%%%%%%%%%
\subsubsection{An elliptic eigenproblem in the whole space}

If $A$ is linear, written as $A(t)=c_At$, then the above transformation reduces to
\begin{equation}
\partial_t v -\frac{1}{L^2}\partial_{yy} v = m(y)v,
\end{equation}
with the same change of variable
\begin{equation}
v(t,y)=u\left(t,L y +c_At\right)\exp\left(\frac{c_A^2 t}{4} +\frac{c_A L y}{2}\right).
\end{equation}

To study the above equation, the following eigenproblem for the operator
\begin{equation}\label{operator-L}
\mathcal{L}=L^{-2}\frac{\upd^2}{\upd y^2}+m,
\end{equation}
where $m$ is as defined in \eqref{m}, will play a crucial role:
\begin{equation}\label{eq:eigenproblem1}
\begin{cases}
-\mathcal{L} \varphi_1 =\lambda\varphi_1 & \text{in }\R, \\
\varphi_1>0 & \text{in }\R, \\
\varphi_1\in W^{2,1}_{\text{loc}}(\R). & 
\end{cases}
\end{equation}
This problem can be solved explicitly, piece by piece with a global $\mathcal{C}^1$ regularity\footnote{Recall
that $W^{2,1}_{\text{loc}}(\R)$ is continuously embedded into $C^{0,\alpha}_{\text{loc}}(\R)$ for any $\alpha\in(0,1)$. 
After a standard bootstrap procedure, we find that the weak solution $\varphi_1$ of
$-L^{-2}\varphi_1''-m\varphi_1=\lambda\varphi_1$ is
in $C^{1,1}_{\text{loc}}(\R)\simeq W^{2,\infty}_{\text{loc}}(\R)$, namely it is of class $C^1$ and its derivative
is locally Lipschitz-continuous.}. 

Unlike a similar eigenvalue problem on a finite interval, an eigenvalue $\lambda$ of $-{\mathcal L}$ with a positive eigenfunction 
is not necessarily unique \revision{\cite[Theorem 1.4]{Berestycki_Ros}}. 
Therefore the right concept of principal eigenvalue has to be defined carefully. According to~\cite[Theorem 1.4]{Berestycki_Ros}, 
the second-order self-adjoint operator~$-\mathcal{L}$ admits a \textit{generalized principal eigenvalue}~$\lambda_1$, which can be defined as the supremum of the set of values of $\lambda$ for which $(\mathcal{L}+ \lambda) \varphi \leq 0$ holds for some positive function $\varphi$. 
Equivalently, it can be defined as the limit of the principal eigenvalues of truncated problems with the Dirichlet boundary condition. By analogy with many previous studies on KPP heterogeneous problems, we expect this generalized eigenvalue to be crucial to accurately predict the spreading speed of solutions for compactly supported initial data. 

\revision{Nonnegative eigenfunctions associated to the generalized principal eigenvalue are referred to as \textit{generalized principal eigenfunctions}.}

Therefore, this subsection is devoted to the computation of the generalized principal eigenvalue $\lambda_1$ \revision{and of an associated generalized principal eigenfunction}.

For later use and convenience, we sum up some of the main results from~\cite{Berestycki_Ros} 
\revision{concerning the definition and basic properties of the principal eigenvalue $\lambda_1$.}

\begin{prop}[{\cite[Theorems~1.4 and~1.7, Proposition~2.3]{Berestycki_Ros}}]\label{prop:eigen}
Define
$$\lambda_1 := \sup \left\{ \lambda \, | \ \exists \varphi \in W^{2,1}_{\text{loc}} (\R) , \ \varphi >0 \ \text{ and } \ (\mathcal{L}+ \lambda) \varphi \leq 0 \right\} \in \mathbb{R}.$$
Then $\lambda_1$ satisfies
\begin{eqnarray*}
\lambda_1 & = & \max \left\{ \lambda \, | \ \exists \varphi \in W^{2,1}_{\text{loc}} (\R) , \ \varphi >0 \ \text{ and } \ - \mathcal{L} \varphi = \lambda \varphi \right\} \\
& = & \inf \left\{ \lambda \, | \ \exists \varphi \in W^{2,1}_{\text{loc}} (\R) \cap L^\infty (\R) , \ \varphi >0 \ \text{ and } \ (\mathcal{L} + \lambda) \varphi \geq 0 \right\}.
\end{eqnarray*}
In particular, if there exists a positive and bounded solution $\varphi$ of $-\mathcal{L} \varphi = \lambda \varphi$ on $\R$, then one must have $\lambda = \lambda_1$.
Furthermore, $$\lambda_1 = \lim_{R \to +\infty} \lambda_1^R,$$
where $\lambda_1^R$ denotes the unique solution of the eigenvalue problem
\begin{equation}
\begin{cases}
-\mathcal{L} \varphi  = \lambda_1^R \varphi & \text{in } (-R,R), \\
\varphi_1>0 & \text{in } (-R,R), \\
\varphi_1 (\pm R) = 0 . & \\ 
\end{cases}
\end{equation}
\end{prop}

As a preliminary step, we note that
\begin{equation}\label{ineq:lambda1}
-r_2 \leq \lambda_1\leq-\max(r_1,r_3).
\end{equation}
Indeed, on the one hand, from our assumption that $r_2 = \max m > \max (r_1,r_3)$, we get by taking $\varphi =1$ as a test function that
$$( \mathcal{L} - r_2) \varphi  \leq 0.$$
Hence, by its definition, $\lambda_1 \geq -r_2$. On the other hand, for any eigenvalue $\lambda>-\max(r_1,r_3)$, any eigenfunction is a sinusoidal function at least on a half-line, and therefore necessarily changes sign. Since Proposition~\ref{prop:eigen} implies that $\lambda_1$ is associated with a positive eigenfunction, we infer that $\lambda_1 \leq -\max (r_1,r_3)$.\\

Hereafter, we will compute the eigenpair $(\lambda_1,\varphi_1)$ in three complementary cases:

\begin{enumerate}[label=(\roman*)]
\item $L>\underline{L}$;
\item $r_1<r_3$ and $L\leq\underline{L}$;
\item $r_1>r_3$ and $L\leq\underline{L}$.
\end{enumerate}
Recall that the threshold $\underline{L}$ is defined by:
\begin{equation}
\underline{L}=
\begin{cases}
0 & \text{if }r_1=r_3, \\
\frac{1}{\sqrt{r_2-\max(r_1,r_3)}}\operatorname{arccot}\left(\sqrt{\frac{r_2-\max(r_1,r_3)}{|r_1-r_3|}}\right) & \text{if }r_1\neq r_3.
\end{cases}
\end{equation}
This definition will become natural in a moment. Among other properties of the eigenpair that will be summarized at the end of this subsection, we will establish the following lemma.
\begin{lem}\label{lem:eigen} The generalized principal eigenvalue $\lambda_1$ satisfies:
\begin{itemize}
\item if $L \leq \underline{L}$, then 
$$\lambda_1 = - \max (r_1, r_3 );$$
\item if $L > \underline{L}$, then $\lambda_1$ is the unique zero of the function
$$
\lambda\in\left(-r_2,\overline{\lambda}\right)\mapsto\cot\left(L\sqrt{r_2+\lambda}\right)-\frac{r_2+\lambda-\sqrt{(r_1+\lambda)(r_3+\lambda)}}{\sqrt{r_2+\lambda}(\sqrt{-r_1-\lambda}+\sqrt{-r_3-\lambda})}
$$
where $\overline{\lambda} = \min \left( -\max (r_1,r_3) , \frac{\pi^2}{L^2} - r_2 \right)$.
\end{itemize}
\end{lem}
The constructions below will involve real constants $C_1,C_2,C_3,C_4,C_5$
with different meanings in each case.

%%%%%%%%%%%%
\subsubsection{Construction of the generalized principal eigenpair in the case $L>\underline{L}$}

In the first case \revision{(i) above,} we look for an eigenfunction of the form:
\begin{equation}\label{caseL_0>underlineL}
\varphi_1(y)=
\begin{cases}
C_1\exp(L\sqrt{-r_1-\lambda_1}y) & \text{if } y\leq 0, 
\\
C_2\sin(L\sqrt{r_2+\lambda_1}y+C_3) & \text{if }0<y<1, \\
C_4\exp(-L\sqrt{-r_3-\lambda_1}y) & \text{if } 1\leq y,
\end{cases}
\end{equation}
with $C_1 >0,C_2,C_3,C_4$ satisfying
\begin{equation}
\begin{cases}
C_1=C_2\sin(C_3),\\
C_2\sin(L\sqrt{r_2+\lambda_1}+C_3)=C_4\exp(-L\sqrt{-r_3-\lambda_1}),\\
C_1\sqrt{-r_1-\lambda_1}=C_2\sqrt{r_2+\lambda_1}\cos(C_3),\\
C_2\sqrt{r_2+\lambda_1}\cos(L\sqrt{r_2+\lambda_1}+C_3)=-C_4\sqrt{-r_3-\lambda_1}\exp(-L\sqrt{-r_3-\lambda_1}),\\
0<C_3<L\sqrt{r_2+\lambda_1}+C_3<\pi.
\end{cases}
\end{equation}
These conditions ensure that the function $\varphi_1$ is $C^1$ and positive. 
\revision{Furthermore, it is clear by construction that $\varphi_1$ is bounded. Therefore, if we can find an eigenfunction of the form \eqref{caseL_0>underlineL}, then by Proposition~\ref{prop:eigen}, $\varphi_1$ is a generalized principal eigenfunction of \eqref{eq:eigenproblem1} and $\lambda_1$ is the generalized principal eigenvalue.}

Fixing without loss of generality $C_1=1$ (which implies $\varphi_1(0)=1$) and rearranging terms thanks to classical trigonometric identities, we end up with:
\begin{equation}
\begin{cases}
C_2=\frac{1}{\sin C_3},\\
\cot C_3 =\sqrt{\frac{-r_1-\lambda_1}{r_2+\lambda_1}} ,\\
C_4=\upe^{L\sqrt{-r_3-\lambda_1}}\frac{\sin(L\sqrt{r_2+\lambda_1}+C_3)}{\sin C_3},\\
\cot\left(L\sqrt{r_2+\lambda_1}\right)=\frac{r_2+\lambda_1-\sqrt{(r_1+\lambda_1)(r_3+\lambda_1)}}{\sqrt{r_2+\lambda_1}(\sqrt{-r_1-\lambda_1}+\sqrt{-r_3-\lambda_1})} ,\\
0<C_3<L\sqrt{r_2+\lambda_1}+C_3<\pi.
\end{cases}
\end{equation}
Note that the last inequality implies $L\sqrt{r_2+\lambda_1}<\pi$, so that $\cot(L\sqrt{r_2+\lambda_1})$ is well-defined indeed. 

It only remains to show that the equation defining implicitly $\lambda_1$ admits indeed a solution smaller than 
$\overline{\lambda}=\min\left(-\max\left(r_1,r_3\right),\frac{\pi^2}{L^2}-r_2\right)$. Note that the value
of~$\overline{\lambda}$ depends on the sign of $L-\frac{\pi}{\sqrt{r_2-\max(r_1,r_3)}}$. Even with the restriction 
$L>\underline{L}$, this quantity can have any sign, since the $\operatorname{arccot}$ function maps $(0,+\infty)$ onto $(0,\pi/2)$,
so that the threshold $\frac{\pi}{\sqrt{r_2-\max(r_1,r_3)}}$ is larger than $2\underline{L}$.

The function
\begin{equation}
\lambda\in\left(-r_2,\overline{\lambda}\right)\mapsto\cot\left(L\sqrt{r_2+\lambda}\right)-\frac{r_2+\lambda-\sqrt{(r_1+\lambda)(r_3+\lambda)}}{\sqrt{r_2+\lambda}(\sqrt{-r_1-\lambda}+\sqrt{-r_3-\lambda})}
\end{equation}
is the sum of two smooth decreasing functions, which can be checked after a lengthy computation which we omit here. When $\lambda\to-r_2$, it tends to $+\infty$.
When $\lambda\to\overline{\lambda}$, it tends to:
\begin{enumerate}
\item $-\infty$ if one of the following two conditions hold true:
\begin{enumerate}
\item $r_1=r_3$, 
\item $r_1\neq r_3$ and $L\geq\frac{\pi}{\sqrt{r_2-\max(r_1,r_3)}}$;
\end{enumerate}
\item $\cot\left(L\sqrt{r_2-\max(r_1,r_3)}\right)-\sqrt{\frac{r_2-\max(r_1,r_3)}{|r_1-r_3|}}$ if $r_1\neq r_3$ and
$L<\frac{\pi}{\sqrt{r_2-\max(r_1,r_3)}}$.
\end{enumerate}
The finite limit in the second case has exactly the sign of
$\underline{L}-L$, by decreasing monotonicity of the $\operatorname{arccot}$ function, and is consequently negative.
Therefore, in all cases, the continuous decreasing function
\begin{equation}\label{eq:fctcase1}
\lambda\in\left(-r_2,\overline{\lambda}\right)\mapsto\cot\left(L\sqrt{r_2+\lambda}\right)-\frac{r_2+\lambda-\sqrt{(r_1+\lambda)(r_3+\lambda)}}{\sqrt{r_2+\lambda}(\sqrt{-r_1-\lambda}+\sqrt{-r_3-\lambda})}
\end{equation}
admits indeed a unique zero.

We conclude that in the case $L > \underline{L}$, we have found a unique positive and bounded eigenfunction of the form \eqref{caseL_0>underlineL}, associated with the eigenvalue $\lambda_1$ defined both as the unique zero of \eqref{eq:fctcase1}, and as the generalized principal eigenvalue from Proposition~\ref{prop:eigen}.

%%%%%%%%%%%%%%%
\subsubsection{Construction of the generalized principal eigenpair in the case $r_1<r_3$ and $L\leq\underline{L}$}
In the second case we look for an eigenpair $(\lambda_1, \varphi_1)$ of the form:
$$\lambda_1 = - r_3,$$
and
\begin{equation}\label{eq:eigen_case2}
\varphi_1(y)=
\begin{cases}
C_1\exp(L\sqrt{-r_1-\lambda_1}y) & \text{if }y \leq 0, \\
C_2\sin(L\sqrt{r_2+\lambda_1}y+C_3) & \text{if }0<y<1, \\
C_4 L (y-1)+C_5 & \text{if }1 \leq y,
\end{cases}
\end{equation}
where $C_1>0,C_2,C_3,C_4,C_5$ satisfy
\begin{equation}\label{case2:conditions}
\begin{cases}
C_1=C_2\sin(C_3),\\
C_2\sin(L\sqrt{r_2+\lambda_1}+C_3)=C_5,\\
C_1\sqrt{-r_1-\lambda_1}=C_2\sqrt{r_2+\lambda_1}\cos(C_3),\\
C_2\sqrt{r_2+\lambda_1}\cos(L\sqrt{r_2+\lambda_1}+C_3)=C_4,\\
0<C_3<L\sqrt{r_2+\lambda_1}+C_3\leq\pi/2.
\end{cases}
\end{equation}
The last condition of~\eqref{case2:conditions} ensures the nonnegativity of~$C_4$, and the positivity of~$C_5$, hence \revision{$\varphi_1>0$} for all $y > 1$. On the other hand, notice that the eigenfunction $\varphi_1$ is no longer bounded here. 
\revision{Therefore, it is not immediate that this is a generalized principal eigenfunction. However, we see from \eqref{ineq:lambda1} that any eigenvalue $\lambda$ associated with a positive eigenfunction satisfies $\lambda \leq -\max(r_1,r_3)=-r_3$. Therefore, if we can show that an eigenfunction $\varphi_1$ of the form \eqref{eq:eigen_case2} with $\lambda_1=-r_3$ exists, then $\lambda_1$ satisfies the condition of Proposition~\ref{prop:eigen}, hence it is the generalized principal eigenvalue of \eqref{eq:eigenproblem1} and $\varphi_1$ is a generalized principal eigenfunction.}

%%\revision{In this case, to state that the eigenvalue is indeed the generalized principal eigenvalue $\lambda_1$ from Proposition~\ref{prop:eigen}, it suffices to notice that, by showing that $-r_3=-\max(r_1,r_3)$ is an admissible eigenvalue for positive eigenfunctions, we have by definition of the generalized principal eigenvalue of Proposition~\ref{prop:eigen} that the inequality $-r_3\leq\lambda_1$ holds, which, together with the converse inequality \eqref{ineq:lambda1}, necessarily implies $\lambda_1=-r_3$. Interestingly, in this case, contrarily to the case $L>\underline{L}$, the eigenfunction $\varphi_1$ is unbounded.}

Fixing without loss of generality $C_1=1$ ({\it i.e.}, $\varphi_1(0)=1$) and rearranging terms as in the previous case, we end up with:
\begin{equation}
\begin{cases}
C_2=\frac{1}{\sin C_3},\\
\cot C_3 = \sqrt{\frac{r_3-r_1}{r_2-r_3}}.\\
C_4=\frac{\sqrt{r_2-r_3}\cos(L\sqrt{r_2-r_3}+C_3)}{\sin C_3},\\
C_5=\frac{\sin(L\sqrt{r_2-r_3}+C_3)}{\sin C_3},\\
0<C_3<L\sqrt{r_2-r_3}+C_3\leq\pi/2.
\end{cases}
\end{equation}
The value of $C_3\in(0,\pi/2)$ is uniquely defined.
It only remains to verify that $L\sqrt{r_2-r_3}+C_3\leq\pi/2$. By assumption, $L\leq\underline{L}<\frac{\pi}{\sqrt{r_2-r_3}}$. Moreover, $\pi/2-C_3\in(0,\pi/2)$. 
Hence both $L\sqrt{r_2-r_3}$ and $\pi/2-C_3$ are in $(0,\pi)$, where the function $\cot$ is well-defined and decreasing, and 
therefore $L\sqrt{r_2-r_3}+C_3\leq\pi/2$ is equivalent to 
\begin{equation}
\cot(L\sqrt{r_2-r_3})\geq\cot\left(\frac{\pi}{2}-C_3\right)=\frac{1}{\cot C_3}=\sqrt{\frac{r_2-r_3}{r_3-r_1}},
\end{equation}
where classical trigonometric identities have been used. This inequality is true by virtue of the assumption
$L\leq\underline{L}$. Note that in the equality case, $C_4=0$.

\subsubsection{Construction of the generalized principal eigenpair in the case $r_1>r_3$ and $L\leq\underline{L}$}
The construction in the third case is done exactly as in the second case, interchanging the roles of $r_1$ and $r_3$.

\subsubsection{Important properties of the generalized principal eigenpairs}\label{sec:important_prop}
To summarize, the elliptic problem~\eqref{eq:eigenproblem1} admits a generalized principal eigenvalue $\lambda_1$ characterized in Proposition~\ref{prop:eigen} and computed in Lemma~\ref{lem:eigen}. The eigenpair $(\lambda_1 ,  \varphi_1)$ also satisfies the following:
\begin{enumerate}
\item if $L>\underline{L}$ (which in particular is always the case when $r_1= r_3$), then $\lambda_1\in(-r_2,-\max(r_1,r_3))$, $\varphi_1\in L^\infty(\R)$, $\varphi_1'(0)=L\sqrt{-r_1-\lambda_1}>0$, $\varphi_1$ decays exponentially at $\pm\infty$;
\item if $r_1<r_3$ and $L\leq\underline{L}$, then $\lambda_1=-r_3$, $\varphi_1'(0)=L\sqrt{r_3-r_1}>0$, $\varphi_1$ grows linearly at $+\infty$ and decays exponentially at $-\infty$;
\item if $r_1>r_3$ and $L\leq\underline{L}$, then $\lambda_1=-r_1$, $\varphi_1'(0) \leq 0$, $\varphi_1$ grows linearly at $-\infty$ and decays exponentially at $+\infty$;
\item the definition of $\lambda_1$ when $L>\underline{L}$ implies, by virtue of 
the implicit function theorem, that $(L,r_1,r_2,r_3)\mapsto\lambda_1$ is smooth in the parameter set $\{L>\underline{L}\}$;
the definition of $\lambda_1$ when $L\leq\underline{L}$ implies that $(L,r_1,r_2,r_3)\mapsto\lambda_1$ is also smooth in the parameter set $\{L\leq\underline{L}, r_1<r_3\}$ and in the parameter set 
$\{L\leq\underline{L}, r_1>r_3\}$; it is therefore at least 
a piecewise-smooth function globally, with at this point an unclear regularity at the interface $\{L=\underline{L}\}$.
\end{enumerate}

%%%%%%%%%%%%%%
\subsection{Proof of Theorem \ref{thm:constant_case}}\label{sec:proof_main}
In this section, for the sake of brevity and clarity, $c^\star$ is defined by anticipation as the quantity on the right-hand side of~\eqref{eq:speed_formula} in Theorem~\ref{thm:constant_case}, \textit{i.e.},
\begin{equation}\label{eq:c*formula}
c^\star =
\begin{cases}
2\sqrt{r_3} & \text{if }c_A<2\sqrt{r_3}, \\
c_A & \text{if }2\sqrt{r_3}\leq c_A\leq 2\sqrt{-\lambda_1} ,\\
F(c_A) & \text{if }2\sqrt{-\lambda_1}<c_A<2\sqrt{r_1}+2\sqrt{-\lambda_1-r_1},\\
2\sqrt{r_1} & \text{if }2\sqrt{r_1}+2\sqrt{-\lambda_1-r_1}\leq c_A
\end{cases}
\end{equation}
where the value of $\lambda_1$ is given in Lemma \ref{lem:eigen} and $F$ is defined by \eqref{eq:def_F}. 
Consequently, $c^\star$ is not, at this point, a notation for a spreading speed. With
this slight abuse of notation, the proof of Theorem \ref{thm:constant_case} reduces to the proof of two inequalities: 
$\overline{c}\leq c^\star$ and $\underline{c}\geq c^\star$.

The proof will be done in four steps:
\begin{enumerate}
    \item[(Step 1)] construct a family of super-solutions showing that $\overline{c}\leq c^\star$ if $c_A\leq 2\sqrt{-\lambda_1}$;
    \item[(Step 2)] construct a family of super-solutions showing that $\overline{c}\leq c^\star$ if $c_A>2\sqrt{-\lambda_1}$;
    \item[(Step 3)] construct a family of sub-solutions showing that $\underline{c}\geq c^\star$ if $c_A\leq 2\sqrt{\max(r_1,r_3)}$ or $c_A\geq 2\sqrt{r_1}+2\sqrt{-\lambda_1-r_1})$;
    \item[(Step 4)] construct a family of sub-solutions showing that $\underline{c}\geq c^\star$ if finally $c_A\in(2\sqrt{\max(r_1,r_3)},2\sqrt{r_1}+2\sqrt{-\lambda_1 - r_1})$.
\end{enumerate}
Bringing all four steps together will immediately end the proof. We note that the open interval $(2\sqrt{\max(r_1,r_3)},2\sqrt{r_1}+2\sqrt{-\lambda_1 - r_1})$ is empty if, and only if, $r_3 < r_1=-\lambda_1$. In such a case, the proof ends at Step 3.

\begin{proof}[Step 1] 
Assume $c_A\leq 2\sqrt{-\lambda_1}$. Define
\begin{equation}
\overline{u}(t,x)=2\min\left(1,\upe^{-\lambda(\max(2\sqrt{r_3},c_A))(x-\max(2\sqrt{r_3},c_A)t-L)}\right),
\end{equation}
where $\lambda(c)=\frac12(c-\sqrt{c^2-4r_3})$. 
This function $\overline{u}(t,x)$ decays exponentially as $x\to +\infty$ and its front propagates at the speed $\max(2\sqrt{r_3},c_A)$. 
Let us show that $\overline{u}$ is a super-solution of \eqref{eq:main}. Note first that $\overline{u}(t,x)$ satisfies the inequality
\begin{equation}\label{ubar-inequality}
\partial_t\overline{u}-\partial_{xx}\overline{u}\geq f(t,x,\overline{u}(t,x))
\end{equation}
in $\{x-\max(2\sqrt{r_3},c_A)t<L\}\cup\{x-\max(2\sqrt{r_3},c_A)t>L\}$. 
Indeed, if $x-\max(2\sqrt{r_3},c_A)t<L$, then since $\overline{u}=2>1$, we have
\[
\partial_t\overline{u}-\partial_{xx}\overline{u}=0> f(t,x,\overline{u}(t,x)),
\]
while, if  $x-\max(2\sqrt{r_3},c_A)t>L$, then since $\lambda(c)(c-\lambda(c))=r_3$, we have
\[
\partial_t\overline{u}-\partial_{xx}\overline{u}=r_3\overline{u}=r(t,x)\overline{u}\geq f(t,x,\overline{u}(t,x)).
\]
Note also that the ``correct angle condition'' ({\it i.e.}, negative derivative gap)  at $x-\max(2\sqrt{r_3},c_A)t=L$ is clearly satisfied. Therefore $\overline{u}$ is a super-solution of \eqref{eq:main}.

Note that $C\overline{u}$ is also a super-solution of \eqref{eq:main} for any $C\geq 1$. Indeed, if $x-\max(2\sqrt{r_3},c_A)t<L$, then since $C\overline{u}=2C>1$, we have
\[
\partial_t\left(C\overline{u}\right)-\partial_{xx}\left(C\overline{u}\right)=0> f(t,x,C\overline{u}(t,x)),
\]
while, if  $x-\max(2\sqrt{r_3},c_A)t>L$, then  
\[
\partial_t\left(C\overline{u}\right)-\partial_{xx}\left(C\overline{u}\right)=r(t,x)\left(C\overline{u}\right)\geq f(t,x,C\overline{u}(t,x))
\]
by the assumption \eqref{ass:KPP}. Also the correct angle condition at $x= \max(2\sqrt{r_3},c_A)t+L$ clearly remains to hold after multiplication by a positive constant $C$.
Now choose~$C$ sufficiently large so that $C\overline{u}$ is initially above the initial condition $u_0$. Then $u(t,x)\leq C\overline{u}(t,x)$ for all $t\geq 0, x\in\R$. 

Recall that $\overline{u}$ decays exponentially as $x\to +\infty$ and its front propagates at the speed $\max(2\sqrt{r_3},c_A)$. As we are assuming $c_A\leq 2\sqrt{-\lambda_1}$, we have $c^\star = \max (2 \sqrt{r_3}, c_A )$.  Hence the inequality $u(t,x)\leq C\overline{u}(t,x)$ proves that~$\overline{c}\leq c^\star$.
\end{proof}

\begin{proof}[Step 2] Assume now that $c_A > 2 \sqrt{-\lambda_1}$, so that in particular $c_A > 2 \sqrt{r_1}$ by \eqref{ineq:lambda1}. Now, following \cite{Girardin_Lam}, we consider the following continuous function:
\begin{equation}\label{super_step2}
\overline{u}(t,x)=
\begin{cases}
2 & \text{if }x\leq ct-\frac{\ln 2}{\lambda(c)} , \\
\upe^{-\lambda(c)(x-ct)} & \text{if }x\in\left(ct-\frac{\ln 2}{\lambda(c)},c_At\right) ,\\
\upe^{-\lambda(c)(c_A-c)t}\upe^{-\frac{c_A(x-c_At)}{2}}\varphi_1\left(\frac{x-c_At}{L}\right) & \text{if }x\geq c_At ,
\end{cases}
\end{equation}
where $c$ is a constant satisfying $2 \sqrt{r_1} \leq c<c_A$, whose value will be specified later, and in this step
$\lambda(c)=\frac12(c-\sqrt{c^2-4r_1})$. The generalized principal eigenfunction $\varphi_1$, which is constructed as in Section~\ref{sec:construc_eigenfct}, 
is appropriately normalized so that $\varphi_1(0)=1$. Since $\varphi_1$ is in both cases $\lambda_1=-r_3$ or $\lambda_1>-r_3$ bounded above by a 
linear function as $x\to+\infty$, it is clear that $\overline{u}$ decays to $0$ as $x\to +\infty$
and that the position of its front propagates at the speed $c$. Let us show that $\overline{u}$ is a super-solution of \eqref{eq:main}. 

First, in $\{x<ct-\ln 2/\lambda(c)\}$, $\overline{u}$ satisfies, by virtue of the assumption \eqref{ass:KPP}:
\begin{equation}
\partial_t \overline{u}-\partial_{xx}\overline{u} = 0\geq f(t,x,\overline{u}(t,x)).
\end{equation}
In $\{ct-\ln 2/\lambda(c)<x<c_At\}$, $\overline{u}$ satisfies:
\begin{equation}
\partial_t \overline{u}-\partial_{xx}\overline{u} = r_1\overline{u}(t,x)\geq f(t,x,\overline{u}(t,x)).
\end{equation}
In $\{x>c_At\}$, it is more convenient to change the viewpoint by defining
\begin{align}
\overline{v}(t,y)=\overline{u}\left(t,L y +c_At\right)\sqrt{L}\upe^{\frac{c_A^2 t}{4} +\frac{c_A L y}{2}}=\sqrt{L}\upe^{-\lambda(c)(c_A-c)t}\upe^{\frac{c_A^2t}{4}}\varphi_1\left(y\right). \nonumber
\end{align}
It satisfies by construction
\begin{equation}
\partial_t\overline{v}-\frac{1}{L^2}\partial_{yy}\overline{v}-m\overline{v} = \left(-\lambda(c)(c_A-c)+\frac{c_A^2}{4}\right)\overline{v}+\lambda_1\overline{v}. 
\end{equation}
Assuming $\lambda_1-\lambda(c)(c_A-c)+\frac{c_A^2}{4}\geq 0$, we obtain $\partial_t\overline{v}-\frac{1}{L^2}\partial_{yy}\overline{v}-m\overline{v}\geq 0$, which gives (back to the original variables)
\begin{equation}
\partial_t\overline{u}-\partial_{xx}\overline{u}\geq r(t,x)\overline{u}(t,x)\geq f(t,x,\overline{u}(t,x)).
\end{equation}
Summarizing, under the assumption that $\lambda_1-\lambda(c)(c_A-c)+\frac{c_A^2}{4}\geq 0$, $\overline{u}$ satisfies
\[
\partial_t \overline{u}-\partial_{xx}\overline{u} \geq f(t,x,\overline{u}(t,x))\quad 
\]
on $\{x<ct-\ln 2/\lambda(c)\}\cup\{ct-\ln 2/\lambda(c)<x<c_At\}\cup\{c_At<x \}$.

In order for $\overline{u}$ to be a super-solution, we also need to verify the angle conditions  ({\it i.e.}, negative derivative gap) at $x=ct-\ln 2/\lambda(c)$ and at $x=c_A t$. The former is easy to check by the definition of $\overline{u}$ in \eqref{super_step2}, so we only need to check
\begin{equation}
\lim_{x-c_At\to 0^-}\partial_x \overline{u}(t,x)\geq\lim_{x-c_At\to 0^+}\partial_x \overline{u}(t,x),
\end{equation}
which reads:
\begin{equation}
-\lambda(c)\geq-\frac{c_A}{2}+\frac{\varphi_1'(0)}{L},
\end{equation}
or equivalently
\begin{equation}
\frac{\varphi_1'(0)}{L}\leq\frac{c_A-c+\sqrt{c^2-4r_1}}{2}.
\end{equation}
Thus, for $\overline{u}$ to be a super-solution, it suffices to verify the following two conditions:
\begin{align}
\frac{\varphi_1'(0)}{L}\leq\frac{c_A-c+\sqrt{c^2-4r_1}}{2}, \nonumber\\
\lambda_1-\lambda(c)(c_A-c)+\frac{c_A^2}{4}\geq 0. \nonumber
\end{align}

Using the properties summarized in Section~\ref{sec:important_prop}, we distinguish two cases:
\begin{enumerate}
\item if $\lambda_1=-r_1$, then by construction of $\varphi_1$, $\varphi_1'(0)\leq 0$,
whence $\varphi_1'(0)/L\leq 0<(c_A-c+\sqrt{c^2-4r_1})/2$, so that
the first condition is always satisfied;
\item if $\lambda_1\neq -r_1$, then by construction of $\varphi_1$, 
$\varphi_1'(0)/L=\sqrt{-r_1-\lambda_1}$ and the first condition rewrites as
\begin{equation}\label{cA-first}
\sqrt{-r_1 - \lambda_1} \leq \frac{c_A}{2} - \lambda (c).
\end{equation}
Squaring both sides, and using that $\left(\frac{c_A}{2}-\lambda(c)\right)^2=\frac{c_A^2}{4}-r_1-(c_A-c)\lambda(c)$ thanks to the equality $\lambda(c)^2-c\lambda(c)+r_1=0$, we find that the first condition implies the second one.
\end{enumerate}
In all cases, in order to show that $\overline{u}$ is a super-solution, all we have to verify is \eqref{cA-first}, or equivalently
\begin{equation}\label{new_cond}
\sqrt{c^2 - 4 r_1} \geq c - c_A   + 2 \sqrt{-\lambda_1 - r_1}.
\end{equation}

Let us first consider the subcase when $c_A \geq 2 \sqrt{r_1} +  2 \sqrt{-\lambda_1 - r_1}$. Then the above inequality is automatically satisfied with $c = 2 \sqrt{r_1}$. 
Thus $\overline{u}$ is a super-solution if we choose  $c = 2 \sqrt{r_1}$. As in the previous case, $C\overline{u}$ is also a super-solution for any~$C>1$. Choose $C$ appropriately large so that $C\overline{u}(0,t)\geq u_0(x)$. Then $u(t,x)\leq C\overline{u}(t,x)$ for all $t\geq 0, x\in\R$. Since $\overline{u}$ decays to $0$ as $x\to +\infty$ and propagates at the speed $c=2\sqrt{r_1}$, the inequality $u(t,x)\leq C\overline{u}(t,x)$ proves that~$\overline{c}\leq 2 \sqrt{r_1}$, where $2 \sqrt{r_1}$ precisely coincides with $c^\star$ in this subcase. Hence $\overline{c}\leq c^\star$ in this subcase.

Next assume that $c_A < 2 \sqrt{r_1} + 2 \sqrt{-\lambda_1 - r_1}$. 
Since $c \geq 2 \sqrt{r_1}$, we have $c - c_A   + 2 \sqrt{-\lambda_1 - r_1}>0$. Hence 
we can square both sides of~\eqref{new_cond} and rewrite it as
\begin{equation}
c\geq \frac12\left(c_A-2\sqrt{-\lambda_1-r_1}+\frac{4r_1}{c_A-2\sqrt{-\lambda_1-r_1}}\right)=F(c_A)
\end{equation}
where the function $F$ is the one defined in \eqref{eq:def_F}. It is a smooth and decreasing function. It maps the interval 
$\left(2\sqrt{-\lambda_1-r_1},2\sqrt{r_1}+2\sqrt{-\lambda_1-r_1}\right]$ onto the interval
$[2\sqrt{r_1},+\infty)$. Moreover its unique fixed point is $F(2\sqrt{-\lambda_1})=2\sqrt{-\lambda_1}$. Therefore
it satisfies $2\sqrt{r_1}\leq F(c')<c'$ for any $c'\in\left(2\sqrt{-\lambda_1},2\sqrt{r_1}+2\sqrt{-\lambda_1-r_1}\right]$. Notice also that $c^\star$ coincides with~$F (c_A)$ in this same interval.

Finally, the partial differential inequality and the angle condition are again satisfied whenever 
$2\sqrt{-\lambda_1} < c_A < 2 \sqrt{r_1} + 2 \sqrt{-\lambda_1 -r_1}$ and for any $c\in [F(c_A), c_A)$, this last interval
being nonempty. Therefore $\overline{u}$ is a super-solution in this case. Arguing precisely as in the previous case, we see that $C\overline{u}$ is also a super-solution for any $C\geq 1$. By choosing $C$ sufficiently large, we have $u(t,x)\leq C\overline{u}(t,x)$ as in the previous case, which proves that $\overline{c}\leq c^\star$ if $c_A> 2\sqrt{-\lambda_1}$.
\end{proof}

Before we proceed, for later use we summarize our analysis in this step in the form of the following lemma, which is presented in a slightly generalized setting:

\begin{lem}\label{lem:super0}
Assume $c_A>2\sqrt{-\lambda_1}$ and that $A(t) = c_A t$ for $t\in[0,T)$ for some $0<T\leq +\infty$. 
Let~$\overline{u}$ be the function defined by~\eqref{super_step2}, with~$c = c^\star$ from~\eqref{eq:c*formula}. Then for any constant $C\geq 1$, $C\overline{u}$ is a super-solution of \eqref{eq:main} 
for $t \in [0,T)$ and $x \in \mathbb{R}$.
\end{lem}

\begin{proof}[Step 3]
Define $\tilde{r}(t,x)$ by \eqref{ass:piecewise_constant_r} with $r_2$ replaced by $r_3$, %%(hence $\tilde{r}(t,x)\leq r(t,x)$) 
and let $\tilde{f}(t,x,u)$ be a function satisfying \eqref{ass:KPP}, $\partial_u\tilde{f}(t,x,0)=\tilde{r}(t,x)$ and $\tilde{f}(t,x,u)\leq f(t,x,u)$. 
Denote by $\tilde{u}(t,x)$ the solution of \eqref{eq:main} with $f$ replaced by $\tilde{f}$ for the same initial data $u_0$. 
Then $\tilde{u}$ is a subsolution of the original \eqref{eq:main}, therefore $\tilde{u}(t,x)\leq u(t,x)$ for all $t\geq 0$, $x\in\R$. 
Thus the speed of $\tilde{u}$ gives a lower bound for the speed of $u$.
If $r_1=r_3$, then $\tilde{r}(t,x)$ is identical to $r_1$, therefore the spreading speed of $\tilde{u}$ is clearly $2\sqrt{r_1}$. 
On the other hand, if $r_1\ne r_3$, then $\tilde{r}(t,x)$ satisfies the assumption of Theorem~\ref{thm:one_interface}, therefore 
the spreading speed of $\tilde {u}$ is again well understood. 
Note that the assertions of Theorem~\ref{thm:one_interface} remain valid even in the case $r_1=r_3$. 
Indeed, if $r_1=r_3$, the cases (a) (ii) and (b) (ii) become empty and all other cases give the value $c^\star=2\sqrt{r_1}=2\sqrt{r_3}$. 
Therefore, in what follows we do not deal with the case $r_1=r_3$ separately. 

First, we consider the case where $c_A \leq 2 \sqrt{\max (r_1,r_3)}$.  
If $r_1\leq  r_3$, then by (b) (i) of Theorem~\ref{thm:one_interface}, the speed of $\tilde{u}$ is $2\sqrt{r_3}$, which coincides with the value of $c^\star$ in \eqref{eq:c*formula}. On the other hand, if $r_1>r_3$, then (a) (i), (ii) of Theorem~\ref{thm:one_interface} apply, which show that the speed of $\tilde{u}$ equals $\max ( 2 \sqrt{r_3}, c_A)$. 
Again this value coincides with $c^\star$ in~\eqref{eq:c*formula}, since $2\sqrt{\max(r_1,r_3)} \leq 2\sqrt{-\lambda_1}$ by virtue of \eqref{ineq:lambda1}. 
Thus we get in this case that $\underline{c} \geq c^\star$.

Next, in the case where $c_A \geq 2 \sqrt{r_1 }+2 \sqrt{-\lambda_1 - r_1}$, then by \eqref{ineq:lambda1} we have $c_A \geq 2\sqrt{r_1} + 2 \sqrt{ \max(r_3 - r_1 ,0)}$. 
Hence, by Theorem~\ref{thm:one_interface} (a) (iii) and (b) (iii), the spreading speed of $\tilde{u}$ is $2 \sqrt{r_1}$. This gives 
$\underline{c} \geq 2 \sqrt{r_1}$, which again coincides with $c^\star$.

To summarize this step, we have found as announced that $\underline{c} \geq c^\star$ if either $c_A \leq 2 \sqrt{\max (r_1 ,r_3)}$ or $c_A \geq 2\sqrt{r_1} + 2 \sqrt{-\lambda_1 - r_1}$.
\end{proof}

\begin{proof}[Step 4] Here we assume that $c_A\in(2\sqrt{\max(r_1,r_3)},2\sqrt{r_1}+2\sqrt{-\lambda_1 - r_1})$. Therefore, throughout this step we have that $c^\star = \min \left(c_A, F (c_A) \right)> 2 \sqrt{r_1}$, where $F$ is given by \eqref{eq:def_F}. 
We refer to Step~2 above for related computations on the function~$F$, including the fact that $F (c') > c'$ for any $c'<2\sqrt{-\lambda_1}$.

Now we consider, for large values of $R>0$,
\begin{equation}
\begin{cases}
-L^{-2}\left(\varphi^R_1\right)''-m\varphi_1^R=\lambda_1^R\varphi_1^R & \text{in }(-R,R), \\
\varphi_1^R(\pm R)=0, & \\
\varphi_1^R>0 & \text{in }(-R,R), \\
\varphi_1^R\in W^{2,1}((-R,R)). & 
\end{cases}
\end{equation}
By the Krein--Rutman theorem, 
the above problem possesses the principal eigenpair $(\lambda_1^R, \varphi_1^R)$, with $\varphi_1^R>0$, which is unique up to multiplication of $\varphi_1^R$ by a constant.
Moreover, according to~\cite{Berestycki_Ros} (see also Proposition~\ref{prop:eigen} above), we have that
$$\lambda_1=\lim_{R\to+\infty}\lambda_1^R.$$
Below, we extend $\varphi_1^R$ in $\R$ by setting $\varphi_1^R=0$ in $\R\backslash(-R,R)$. We consider the following continuous function:
\begin{equation}\label{subsolstep4}
\underline{u}(t,x)=
\begin{cases}
0 & \text{if }x<\frac{\ln S}{\eta}+x_0-2R', \\[3pt]
\iota O(x) \hspace{-85pt} &  \text{if }x\in\left[\frac{\ln S}{\eta}+x_0-2R',\frac{\ln S}{\eta}+x_0-\frac{R'}{3}\right), \\[3pt]
\iota\frac{\sigma}{2} & \text{if }x\in\left[\frac{\ln S}{\eta}+x_0-\frac{R'}{3},\frac{\ln S}{\eta}+x_0+ct+x_1\right), \\[3pt]
\iota P(t,x) & \text{if }x\in\left[\frac{\ln S}{\eta}+x_0+ct+x_1,X(t)\right), \\[3pt]
\iota Q(t,x) & \text{if }x\geq X(t),
\end{cases}
\end{equation}
where
\begin{equation}\label{eq:OPQ}
\begin{split}
O(x) & = \sigma\sin\left(\frac{\pi}{2R'}\left(x-\frac{\ln S}{\eta}-x_0+2R'\right)\right), \\
P(t,x) & = \upe^{-\lambda(c)(x-ct-x_0)}-S\upe^{-(\lambda(c)+\eta)(x-ct-x_0)}, \\
Q(t,x) & = \gamma \upe^{-\lambda(c)(c_A-c)t}\upe^{-\frac{c_A(x-c_At)}{2}}\varphi_1^R\left(\frac{x-c_At}{L}\right), 
\end{split}
\end{equation}
with $\eta>0$, $S>0$, $R'>\frac{\pi}{2\sqrt{r_1}}$, $\iota\in(0,1]$, 
$\sigma\in\left(0,2\max_{x\in\R}P(0,x)\right)$, where $\lambda(c)=\frac12(c-\sqrt{c^2-4r_1})$ and  $c\in(2\sqrt{r_1},c_A)$, which is a nonempty interval here, $x_0\in\R$, $\gamma>0$, $R>1$, 
with the principal eigenfunction $\varphi_1^R$ appropriately normalized so that $\varphi_1^R(0)=1$. 
The values of $x_1$, $X(t)$ are chosen to make $\underline{u}$ continuous at $x=\frac{\ln S}{\eta}+x_0+ct+x_1$ and at $x=X(t)$. More precisely:
\begin{itemize}
    \item $x_1>0$ is the well-defined smallest zero of $x\mapsto\frac{\sigma}{2}-P\left(0,x+x_0+\frac{\ln S}{\eta}\right)$ (note that it satisfies indeed, for
    any $t\geq 0$, $\frac{\sigma}{2}=P\left(t,\frac{\ln S}{\eta}+x_0+ct+x_1\right)$ and note moreover that $x_1\to0$ as $\sigma\to 0$);
    \item $X(t)\in(c_At-RL,c_At)$ is, for each $t\geq 0$, the smallest zero of $x\mapsto P(t,x)-Q(t,x)$.
\end{itemize}
The values of $\iota$, $\sigma$, $\eta$, $S$, $x_0$, $R$ and $\gamma$ will be specified later on. 
Typically, $1/\iota$, $1/\sigma$, $1/\eta$, $1/\gamma$, $S$, $-x_0$ and $R$ are large positive numbers.  We will also check later that $c$ can be chosen arbitrarily close to $c^\star$.

The existence and uniqueness of $X(t)$, provided the parameters are appropriately chosen, will be verified in Lemma~
\ref{lem:interface_sub-solution}. In particular, we point out that $r(t,x)=r_1$ if $x\leq X(t)$ due to $X(t) \leq c_A t$. Lemma~\ref{lem:interface_sub-solution} will also establish that $\frac{\ln S}{\eta}+x_0 < \inf_{t \geq 0} X(t) - ct$. Then, choosing $\sigma$ small enough so that $\frac{\ln S}{\eta} + x_0 + x_1 \leq \inf_{t \geq 0} X(t) -ct$, we get that each interval in the above definition of $\underline{u}$ is nonempty, which in turn means that $\underline{u}$ is continuous.

Note that, for each $t\geq 0$, the support of $\underline{u}(t,x)$ is the interval $[\frac{\ln S}{\eta}+x_0-2R', c_At+RL]$, and that the main front of $\underline{u}(t,x)$, which is represented by the above function $P(t,x)$, propagates at the speed $c$.
Behind the front, $\underline{u}(t,x)=\iota\sigma/2$ except near the left endpoint of its support. Moreover, $\underline{u}$ is globally
proportional to $\iota$ and this parameter can be interpreted as an amplitude parameter.

In order for $\underline{u}$ to be a potential sub-solution, first the following angle condition 
({\it i.e.}, positive derivative gap) needs to be satisfied:
\begin{equation}
\lim_{x\to X(t)^-}\partial_x \underline{u}(t,x)\leq\lim_{x\to X(t)^+}\partial_x \underline{u}(t,x).
\end{equation}
This will be a consequence of Lemma~\ref{lem:interface_sub-solution} below, whose first two statements imply
\begin{equation}\label{angle}
\lim_{x \to X(t)^-} \partial_x \underline{u} (t,x) < 0 < \lim_{x \to X(t)^+} \partial_x \underline{u} (t,x).
\end{equation}
We point out that the other angle conditions at the interfaces of the other subintervals are more straightforward, hence we skip their proof for the sake of brevity.

Next, we check that $\underline{u}$ satifies the desired differential inequalities on each subdomain. First, in the region $\left\{x<\frac{\ln S}{\eta}+x_0-2R'\right\}$, $\underline{u}$ is identical to $0$, which directly yields $\partial_t\underline{u}-\partial_{xx}\underline{u}\leq f(t,x,\underline{u}(t,x))$.

In $\left\{x\in\left[\frac{\ln S}{\eta}+x_0-2R',\frac{\ln S}{\eta}+x_0-\frac{R'}{3}\right)\right\}$, the function~$\underline{u}$ satisfies:
\begin{equation}
 \partial_t\underline{u}-\partial_{xx}\underline{u}%=\frac{\pi^2}{4R'^2}\iota\sigma\sin\left(\frac{\pi}{2R'}x\right)
=\frac{\pi^2}{4R'^2}\underline{u}.
\end{equation}
Recall that $\frac{\ln S}{\eta} + x_0 - \frac{R'}{3} < 0 \leq c_A t$ for $t \geq 0$ by Lemma~\ref{lem:interface_sub-solution}, and thus we have that $\partial_u f (t,x,0) = r_1$ here. By virtue of~\eqref{ass:KPP} and $\frac{\pi^2}{4R'^2}<r_1$, we can assume that $\sigma$ is so small that
\begin{equation}
	f(t,x,v) \geq \frac{\pi^2}{4 R'^2} v,
\end{equation}
for any $x\leq \frac{\ln S}{\eta}+x_0-\frac{R'}{3}$ and $v \in [0, \sigma]$. Remarking that $\underline{u}\leq\iota\sigma\leq\sigma$ in  $\left\{ x \in\left[\frac{\ln S}{\eta}+x_0-2R',\frac{\ln S}{\eta}+x_0-\frac{R'}{3}\right) \right\}$ (by virtue of $\iota\leq 1$), we deduce:
\begin{equation}
    \partial_t\underline{u}-\partial_{xx}\underline{u}\leq f(t,x,\underline{u}(t,x)).
\end{equation}

In $\left\{x\in\left[\frac{\ln S}{\eta}+x_0-\frac{R'}{3},\frac{\ln S}{\eta}+x_0+ct+x_1\right)\right\}$, again and quite similarly,
\begin{equation}
    \partial_t\underline{u}-\partial_{xx}\underline{u} =0 \leq f(t,x,\underline{u}(t,x)),
\end{equation}
provided $\sigma$ is small enough, uniformly in $\iota\leq 1$.

Next, in $\left\{\frac{\ln S}{\eta}+x_0+ct+x_1<x<X(t),\ t>0\right\}$, the function~$\underline{u}$ satisfies:
\begin{equation}
\begin{aligned}
& \qquad \partial_t\underline{u}-\partial_{xx}\underline{u} \\
& =\iota r_1\upe^{-\lambda(c)(x-ct-x_0)}-\iota S\left(r_1-\eta^2+\eta\sqrt{c^2-4r_1}\right)\times \upe^{-(\lambda(c)+\eta)(x-ct-x_0)} \\
& = r_1\underline{u}(t,x)-\iota S\eta \left(\sqrt{c^2-4r_1}-\eta\right) \times \upe^{-(\lambda(c)+\eta)(x-ct-x_0)} \\
& \leq \underline{u}(t,x)\upe^{-\eta(x-ct-x_0)}\left(\left(r_1-\frac{f(t,x,\underline{u}(t,x))}{\underline{u}(t,x)}\right)\upe^{\eta(x-ct-x_0)}-S\eta\left(\sqrt{c^2-4r_1}-\eta\right)\right) \\
& \quad +f(t,x,\underline{u}(t,x)).
\end{aligned}
\end{equation}
We claim that, denoting $g(t,x,u)=r_1-f(t,x,u)/u$, the function
\begin{equation}\label{claim:negative}
(t,x)\mapsto g(t,x,\underline{u}(t,x))\upe^{\eta(x-ct-x_0)} - S \eta \left( \sqrt{c^2 - 4 r_1 } - \eta \right)
\end{equation}
is negative in~$\left\{\frac{\ln S}{\eta}+x_0+ct<x<X(t),\ t>0\right\}$ for $\eta$ small enough and $S$ large enough (depending on $\eta$). If this claim holds true, then $\underline{u}$ indeed satisfies the desired differential inequality in this subdomain.

Let us therefore verify this claim. First recall that $X(t) < c_A t$. Then, by the assumption~\eqref{ass:KPP},
\begin{equation}
r_1 u \geq f(t,x,u) \geq r_1 u - M u^2\quad\text{for all }(t,x)\in\{x<c_At\}\text{ and all }u\geq 0,
\end{equation}
so that
\begin{equation}
0 \leq g(t,x,\underline{u}) \leq M \underline{u} \quad \text{for all } (t,x) \in \{x < c_A t\},
\end{equation}
for some $M>0$. It follows that
\begin{align*}
 & g(t,x,\underline{u}(t,x))\upe^{\eta(x-ct-x_0)} - S \eta \left( \sqrt{c^2 - 4 r_1 } - \eta \right) \\
& \quad \leq M \underline{u} (t,x) \upe^{\eta(x-ct-x_0)} - S \eta \left( \sqrt{c^2 - 4 r_1 } - \eta\right) \\
& \quad \leq M \iota \upe^{-(\lambda(c) - \eta) (x -ct-x_0)}  - S \eta \left(\sqrt{c^2 - 4 r_1} - \eta \right)\\
& 
\quad \leq M \iota \upe^{-(\lambda(c) - \eta) \ln S/\eta}  - S \eta \left(\sqrt{c^2 - 4 r_1} - \eta \right)
\end{align*}
in~$\left\{\frac{\ln S}{\eta}+x_0+ct<x<X(t),\ t>0\right\}$. Therefore, provided that $S>1$ and $\eta<\lambda(c)$, we get that
\begin{align*}
g(t,x,\underline{u}(t,x))\upe^{\eta(x-ct-x_0)} - S \eta ( \sqrt{c^2 - 4 r_1 } - \eta) 
& \leq S \left( M\iota S^{- \frac{\lambda(c)}{\eta}}  -  \eta (\sqrt{c^2 - 4 r_1} - \eta) \right) \\
& \leq S \left( M\iota S^{- 1}  -  \eta (\sqrt{c^2 - 4 r_1} - \eta) \right) .
\end{align*}
We deduce that if 
\begin{equation}
\eta<\min(\lambda(c),\sqrt{c^2-4r_1})\quad\text{and}\quad
S>\max\left(1,\frac{M}{\eta\left(\sqrt{c^2-4r_1}-\eta\right)}\right),
\end{equation}
then $\underline{u}$ is a sub-solution in $\left\{\frac{\ln S}{\eta}+x_0+ct<x<X(t),\ t>0\right\}$,
independently of the exact values of $\iota \in (0,1]$, $R>1$ and $x_0\in\R$.

Finally, for $\left\{x>X(t)\right\}$, we introduce a new function
\begin{equation}
\underline{v}(t,y) =\underline{u}(t,Ly+c_At)\sqrt{L}\upe^{\frac{c_A^2 t}{4} +\frac{c_A L y}{2}}
= \iota\gamma\sqrt{L}\upe^{-\lambda(c)(c_A-c)t}\upe^{\frac{c_A^2t}{4}}\varphi_1^R\left(y\right).
\end{equation}
It satisfies by construction
\begin{equation}
\partial_t\underline{v}-\frac{1}{L^2}\partial_{yy}\underline{v}-m\underline{v}=\left(-\lambda(c)(c_A-c)+\frac{c_A^2}{4}+\lambda_1^R\right)\underline{v}.
\end{equation}
Assume that
\begin{equation}\label{eq:necessary0}
\lambda_1-\lambda(c)(c_A-c)+\frac{c_A^2}{4}< 0,
\end{equation}
set $\delta=\frac12|\lambda_1-\lambda(c)(c_A-c)+\frac{c_A^2}{4}|$ and recall $\lim_{R\to+\infty}\lambda_1^R=\lambda_1$.
Assume now that~$R$ is so large that $\lambda_1^R-\lambda(c)(c_A-c)+\frac{c_A^2}{4}<-\delta$, whence
$\partial_t\underline{v}-\frac{1}{L^2}\partial_{yy}\underline{v}-m\underline{v}\leq -\delta\underline{v}$, 
which gives (back to the original variables)
\begin{equation}
\partial_t\underline{u}-\partial_{xx}\underline{u}\leq (r(t,x)-\delta)\underline{u}(t,x).
\end{equation}
Using again~\eqref{ass:KPP} which ensures the continuity of $(t,x,u)\mapsto f(t,x,u)/u$, and all other parameters being fixed, we can assume that $\gamma>0$ is so small that 
\begin{equation}
r(t,x)-\delta\leq \inf_{(t,x,v)\in\R\times\R\times \left[0,\gamma e^{\frac{c_A R L}{2}} \max \varphi_1^R \right]}\frac{f(t,x,v)}{v} .
\end{equation}
Remarking that $\underline{u}\leq \gamma e^{\frac{c_A R L}{2}} \max \varphi_1^R $ in $\left\{x>X(t)\right\}$ (by virtue of $\iota\leq1$ and $X(t) \geq c_A  t - R L$), we deduce:
\begin{equation}
\partial_t\underline{u}-\partial_{xx}\underline{u}\leq f(t,x,\underline{u}(t,x)).
\end{equation}
We observe that, $c<c_A \in(2\sqrt{\max(r_1,r_3)},2\sqrt{r_1}+2\sqrt{-\lambda_1 - r_1})$ being given, the above necessary condition~\eqref{eq:necessary0}, \textit{i.e.}, $\lambda_1-\lambda(c)(c_A-c)+\frac{c_A^2}{4}< 0$, is equivalent to $c<F(c_A)$. Finally, we have shown that the function defined in~\eqref{subsolstep4} is a sub-solution, and that its speed~$c$ can be chosen arbitrarily close to $\min \left( c_A , F (c_A) \right) $, which coincides with~$c^\star$ in this parameter range.\\

To conclude Step~4, it remains to prove important properties of $X(t)$.

%%%%
\begin{lem}\label{lem:interface_sub-solution}
Let $P(t,x),Q(t,x)$ be as in \eqref{eq:OPQ}. 
For any positive values of $c,\eta,S,R,\gamma$ with $c<c_A$, there exists $x_0\in\R$ such that 
the equation $P(t,x)=Q(t,x)$, or more precisely, 
\begin{equation}\label{eq:interface_sub-solution}
\begin{split}
& \upe^{-\lambda(c)(x-ct-x_0)}-S\upe^{-(\lambda(c)+\eta)(x-ct-x_0)}\\
&\hspace{60pt} = \gamma\upe^{-\lambda(c)(c_A-c)t}\upe^{-\frac{c_A(x-c_At)}{2}}\varphi_1^R\left(\frac{x-c_At}{L}\right)
\end{split}
\end{equation}
admits for all $t\geq 0$ an isolated solution $X(t)\in\R$ such that:
\begin{enumerate}[label=(\alph*)]
\item $\partial_x P(t,X(t))<0$; 
\item $\partial_x Q(t,X(t))>0$;
\item $c_At-RL<X(t)<c_At$;
\item $\frac{\ln S}{\eta}< \inf_{ t \geq 0 } X(t)-ct-x_0$.
\end{enumerate}
Moreover, the trajectory $t\mapsto X$(t) satisfies:
\begin{enumerate}[label=(\alph*),start=5]
  % \addtocounter{enumi}{5}
 % \setcounter{enumi}{4}
\item $X\in\mathcal{C}^1([0,+\infty),(-RL,+\infty))$;
\item $X(t)=c_At+O(1)$ as $t\to+\infty$.
\end{enumerate}
\end{lem}

\begin{proof}
For any $t\geq 0$, the function 
$P:x\mapsto \upe^{-\lambda(c)(x-ct-x_0)}-S\upe^{-(\lambda(c)+\eta)(x-ct-x_0)}$ is unimodal, increasing on the left of
$x=ct+x_0+\frac{1}{\eta}\ln\left(\frac{S(\lambda(c)+\eta)}{\lambda(c)}\right)$ and decreasing on its right.
Therefore, if 
\begin{equation}\label{eq:x0add}
x_0<-RL-\frac{1}{\eta}\ln\left(\frac{S(\lambda(c)+\eta)}{\lambda(c)}\right),
\end{equation}
then, since $c<c_A$, $x\mapsto P(t,x)$ is, at any $t\geq 0$, decreasing in $[c_At-RL,c_At]$. 
In particular, statement \textit{(a)} will be a consequence of \textit{(c)}.

For any $t\geq 0$, the function
\begin{equation}
y\in[-R,+R]\mapsto\gamma\upe^{-\lambda(c)(c_A-c)t}\upe^{-\frac{c_A L y}{2}}\varphi_1^R\left(y\right)
\end{equation}
\revision{has the following} derivative:
\begin{equation}
y\in[-R,+R]\mapsto\gamma\upe^{-\lambda(c)(c_A-c)t}\upe^{-\frac{c_A Ly}{2}}\left(-\frac{c_A L}{2}\varphi_1^R(y)+ \left( \varphi_1^R \right) '(y)\right).
\end{equation}
At $y=-R$, by virtue of the Hopf lemma (\textit{i.e.}, $\left( \varphi_1^R \right) '(-R)>0$), this derivative is positive.
Similarly, at $y=+R$, the derivative is negative.
Therefore, by virtue of the intermediate value theorem, there exists $r\in(-R,R)$ such that 
$y\mapsto\gamma\upe^{-\lambda(c)(c_A-c)t}\upe^{-\frac{c_ALy}{2}}\varphi_1^R\left(y\right)$
is increasing in $(-R,-R+r)$. Without loss of generality, up to reducing $r$ we may assume that $-R+r <0$.

Consequently, the function 
\begin{equation}
x\in[-R L +c_At,-(R-r)L+c_At]\mapsto Q(t,x)
\end{equation} 
is increasing. 
Statement \textit{(b)} of Lemma~\ref{lem:interface_sub-solution} follows from this monotonicity property, along with the fact to be established below that $X(t)$ belongs to the interval $[-R L + c_A t, -(R-r) L + c_A t]$.

In view of the monotonicities in $x$ of each side of the equality \eqref{eq:interface_sub-solution} (decreasing on the left, increasing on the right), if $-x_0$ is so large that
\begin{equation}\label{-R+r}
\upe^{\lambda(c)((R-r)L+x_0)}<\gamma\upe^{\frac{c_A}{2}(R-r)L}\varphi_1^R(-R+r),
\end{equation}
then by the intermediate value theorem, at $t=0$ there is a (unique) solution~$X(0)$ of \eqref{eq:interface_sub-solution} in $(-RL,-(R-r)L)$.
By the implicit function theorem, this solution can be extended in a continuously differentiable way in an open time interval around $t=0$. 
To show that $X(t)$ can be extended globally, for all $t\geq 0$, it suffices to prove that $X(t)<c_At-(R-r)L$ for any $t$ such that $X(t)$ is well-defined. 
In order to verify this inequality, recall that $X(t)$ satisfies $P(t,X(t))=Q(t,X(t))$ and that $P(t,x)$ (respectively $Q(t,x)$) is monotone decreasing (respectively increasing) in $x$. Therefore all we need to show is that $Q(c_At-(R-r)L,t)>P(c_At-(R-r)L,t)$ for such $t$, or equivalently,
\begin{eqnarray*}
&& \gamma\upe^{-\lambda(c)(c_A-c)t}\upe^{\frac{c_A}{2}(R-r)L}\varphi_1^R(-R+r)  \\
& > & \upe^{-\lambda(c)(c_At-(R-r)L-ct-x_0)}-S\upe^{-(\lambda(c)+\eta)(c_At-(R-r)L-ct-x_0)} .
\end{eqnarray*}
This follows directly from:
\begin{equation}
\begin{aligned}
\gamma\upe^{\frac{c_A}{2}(R-r)L}\varphi_1^R(-R+r) & >\upe^{\lambda(c)((R-r)L+x_0)} \\
& >\upe^{\lambda(c)((R-r)L+x_0)}-S\upe^{(\lambda(c)+\eta)((R-r)L+x_0)-\eta(c_A-c)t},
\end{aligned}
\end{equation}
where we used~\eqref{-R+r}.

By construction, we have that $c_A t - R L < X (t) < c_A t$, \textit{i.e.}, \textit{(c)} holds true. As explained above, statement~\textit{(a)} and~\textit{(b)} follow, and so does~\textit{(f)}. The differentiability property~\textit{(e)} follows from the implicit function theorem. Finally, \textit{(d) i.e} the estimate $\frac{\ln S}{\eta}< \inf_{t \geq 0 } X(t)-ct-x_0$ results from a direct calculation:
\begin{equation}
X(t)>c_At-RL>c_At+x_0+\frac{\ln S}{\eta}+\frac{\ln(1+\eta/\lambda(c))}{\eta}>ct+x_0+\frac{\ln S}{\eta},
\end{equation}
where we used~\eqref{eq:x0add}. This completes the proof of Lemma~\ref{lem:interface_sub-solution}.
\end{proof}

To summarize Step~4, thanks to Lemma \ref{lem:interface_sub-solution}, the key partial differential inequality and the angle condition \eqref{angle} are verified for any 
$c\in\left(2\sqrt{r_1},\min\left(c_A,F\left(c_A\right)\right)\right)$, with an amplitude parameter
$\iota\in(0,1]$ whose value can be chosen arbitrarily small. 

Since $\underline{u}$ is compactly supported, choosing $\iota$ small enough, $\underline{u}$ is below the solution~$u$ at time $t=1$.
Hence $\underline{u}$ is a sub-solution for all $t\geq 1$ and this proves that~$\underline{c}\geq c^\star$ if $c_A\in(2\sqrt{\max(r_1,r_3)},2\sqrt{r_1}+2\sqrt{-\lambda_1-r_1})$.
\end{proof}

Again, for later use we state the following lemma which sums up our construction of a sub-solution in this last step.
%%%%%
\begin{lem}\label{lem:sub0}
Assume 
that $c_A\in(2\sqrt{\max(r_1,r_3)},2\sqrt{r_1}+2\sqrt{-\lambda_1 - r_1})$ and 
that $A(t) = c_A t$. Let~$c^\star$ and $\underline{u}$ respectively be defined by~\eqref{eq:c*formula} and~\eqref{subsolstep4} with $c= c^\star - \varepsilon$. 
Then for any $\varepsilon >0$ sufficiently small, there exist positive constants $R', \sigma, \eta, S, R, x_0, \gamma$ such that, for all $\iota \in (0,1]$, the function~$\underline{u}$ is a sub-solution of \eqref{eq:main} 
%%$$\partial_t \overline{u} \leq  \partial_{xx} \overline{u} +  r(t,x) \overline{u} -  M u^2.$$
for $t >0$ and $x \in \mathbb{R}$ whose front propagates at the speed $c$. For each $t\geq 0$, the support of $\underline{u}(t,x)$ is compact, and 
$\underline{u}$ is proportional to the amplitude parameter $\iota$.  
\end{lem}

%%%%%%%%%%%%%%%
\subsection{Proof of Corollaries \ref{cor:slow_patch} and \ref{cor:fast_patch}}

The results follow from simple comparison arguments. 
As the proof proceeds exactly the same way, we replace the constant~$L$ in~\eqref{ass:piecewise_constant_r} by an arbitrary positive and continuous function~$L(t)$.

\begin{proof}[Proof of Corollary \ref{cor:slow_patch}] 
Assuming that $\sup_{t \geq 0} \frac{A(t) + L(t)}{L(t)} \leq 2 \sqrt{r_3}$,
we have that
$$\underline{r} (t,x) \leq r(t,x) \leq \overline{r} (t,x),$$
where
\begin{equation*}
\underline{r} :(t,x)\mapsto
\begin{cases}
\min \left( r_1, r_2, r_3 \right) & \text{if }x< 2 \sqrt{r_3} t, \\
r_3 & \text{if } 2 \sqrt{r_3} t \leq x,
\end{cases}
\end{equation*}
\begin{equation*}
\overline{r} :(t,x)\mapsto
\begin{cases}
\max \left( r_1, r_2, r_3 \right) & \text{if }x< 2 \sqrt{r_3} t, \\
r_3 & \text{if } 2 \sqrt{r_3} t \leq x .
\end{cases}
\end{equation*}
Let $\underline{u}$ and $\overline{u}$ denote the solutions of \eqref{eq:main} corresponding to $\underline{r}$ and $\overline{r}$, respectively. Then by the comparison principle, 
\[
\underline{u}(t,x)\leq u(t,x)\leq \overline{u}(t,x).
\]
According to Theorem \ref{thm:one_interface}, $\underline{u}$ and $\overline{u}$ both spread with speed $2 \sqrt{r_3}$. Consequently, we have $\underline{c} = \overline{c} = 2\sqrt{r_3}$ under the assumptions of Corollary~\ref{cor:slow_patch}.
\end{proof}

\begin{proof}[Proof of Corollary \ref{cor:fast_patch}] The proof proceeds similarly. Notice that if $\inf_{t \geq 0}\frac{A(t)}{t} \geq 2 \sqrt{r_1} + 2 \sqrt{r_2 -r_1}$, then
$$\underline{r} (t,x) \leq r(t,x) \leq \overline{r} (t,x),$$
where
\begin{equation*}
\underline{r} :(t,x)\mapsto
\begin{cases}
r_1  & \text{if }x< 2 \sqrt{r_1} + 2 \sqrt{r_2 -r_1}, \\
\min \left( r_1, r_2, r_3 \right)  & \text{if } 2 \sqrt{r_1} + 2 \sqrt{r_2 -r_1} \leq x,
\end{cases}
\end{equation*}
\begin{equation*}
\overline{r} :(t,x)\mapsto
\begin{cases}
r_1 & \text{if }x<  2 \sqrt{r_1} + 2 \sqrt{r_2 -r_1}, \\
\max \left( r_1, r_2, r_3 \right) & \text{if } 2 \sqrt{r_1} + 2 \sqrt{r_2 -r_1} \leq x .
\end{cases}
\end{equation*}
Recall that $\max \left( r_1, r_2,r_3 \right) = r_2$. Applying Theorem~\ref{thm:one_interface} and using the comparison principle as in the proof of Corollary \ref{cor:slow_patch}, we find that $\underline{c} = \overline{c} =  2\sqrt{r_1}$.
\end{proof}

%\begin{rem}\label{rem:one_interface_as_corollary}
%Having proved Theorem \ref{thm:constant_case}, we can deduce Theorem \ref{thm:one_interface} as
%a corollary. Even though our method of proof does require $L>0$ and $r_2>\max(r_1,r_3)$, so that Theorem
%\ref{thm:one_interface} is not a mere special case, it can indeed be deduced by comparisons
%similar to those used above to prove Corollaries \ref{cor:slow_patch} and \ref{cor:fast_patch}.
%Indeed, assume that $r$ satisfies the assumptions of Theorem~\ref{thm:one_interface}. Then $\overline{r}=r+\varepsilon \mathbf{1}_{\{x-c_At\in[0,1]\}}$ falls under the scope of Theorem~\ref{thm:constant_case}, which brings forth a super-solution that 
%subsequently leads to an estimate on $\overline{c}$, whereas $\underline{r}=r-\varepsilon \mathbf{1}_{\{x-c_At>1\}}$ 
%brings forth a sub-solution that subsequently leads to an estimate on $\underline{c}$. Since $\underline{L}\to+\infty$ as $r_2\to\max(r_1,r_3)$, we have $\underline{L}>1$ if $\varepsilon$ is chosen sufficienty small.  This ends the arguments.
%\end{rem}

%%%%%%%%%%%%%%%%%
\subsection{Proof of Theorem \ref{thm:intermediate_case_large_oscillations}} 

Here we prove that if the moving speed of the heterogeneity alternates between two values very slowly, then the minimal and maximal spreading speeds $\underline{c}$ and~$\overline{c}$ may differ. We will prove this result by constructing appropriate super- and sub-solutions.
 
The values of $r_1,r_2,r_3,L$ being fixed, we deduce from the statement of Theorem~\ref{thm:constant_case} a value
for $\lambda_1$ which does not depend on $c_A$. By assumption, $\lambda_1\neq-r_1$, and $2\sqrt{-\lambda_1}<c_{A,1}<c_{A,2}<2\sqrt{r_1}+2\sqrt{-\lambda_1-r_1}$. For each $i\in\{1,2\}$, we define
\begin{equation}\label{c*_12}
    c^\star_i = F(c_{A,i}) = \frac12\left(c_{A,i}-2\sqrt{-\lambda_1-r_1}+\frac{4r_1}{c_{A,i}-2\sqrt{-\lambda_1-r_1}}\right). %%< c_{A,2}.
\end{equation}
As one can check, we have
\begin{equation}\label{c-comparison}
2\sqrt{r_1}<c^\star_2<c^\star_1 < 2\sqrt{-\lambda_1}<c_{A,1}<c_{A,2}.
\end{equation}
Note that the above definition of $c^\star_i\,(i=1,2)$ agrees with that of $c^\star$ in \eqref{eq:c*formula} for the case $2\sqrt{-\lambda_1}<c_A<2\sqrt{r_1}+2\sqrt{-\lambda_1-r_1}$, where $c_A$ will be replaced by $c_{A,1}$ or $c_{A,2}$ in the later arguments. This enables us to apply Lemma~\ref{lem:super0} and Lemma~\ref{lem:sub0} in the construction of super- and sub-solutions.

Hereafter, we define 
\begin{equation}\label{eq:r_abuse}
\tilde{r} : z \mapsto
\begin{cases}
r_1 & \text{if } z< 0, \\
r_2 & \text{if } 0 \leq z < L , \\
r_3 & \text{if } L \leq z .
\end{cases}
\end{equation}
Thus the function $r(t,x)$ in \eqref{ass:piecewise_constant_r} is expressed as $r(t,x)=\tilde{r}(x-A(t))$.

The aim of this section is to prove that, under the assumptions of Theorem~\ref{thm:intermediate_case_large_oscillations}, the minimal and maximal spreading speeds of the solution~$u$ of \eqref{eq:main} satisfy
\begin{equation*}
    \underline{c} \leq c^\star_2 + 2 \varepsilon < c^\star_1 - 2 \varepsilon \leq \overline{c},
\end{equation*}
for all sufficiently small $\varepsilon >0$, which implies $\underline{c} \leq c^\star_2 < c^\star_1  \leq \overline{c}$.

Before starting the proof, we note that the function $A(t)$ defined in \eqref{A(t)} satisfies
\begin{equation}
c_{A,1}\hspace{1pt}t \leq A(t)<c_{A,2}\hspace{1pt}t\ \ \hbox{for all } t\geq 0.
\end{equation}
In particular, we have $A(t_{2n+1})<c_{A,2}\hspace{1pt} t_{2n+1}$ ($n=0,1,2,3,\ldots$).

We will use basically the same super- and sub-solutions constructed in the proof of Theorem~\ref{thm:constant_case} with minor modifications. Since we are assuming $2\sqrt{-\lambda_1}<c_{A,1}<c_{A,2}<2\sqrt{r_1}+2\sqrt{-\lambda_1-r_1}$, the functions $\overline{u}$ defined in \eqref{super_step2} in Step~2 of Section~\ref{sec:proof_main} and $\underline{u}$ defined in \eqref{subsolstep4}  in Step~4 will be relevant. 

First, in the time interval $I_{2n+1}=[t_{2n+1},t_{2n+2})$, we use the function $\overline{u}$ defined in \eqref{super_step2} with $c_A$ replaced by $c_{A,2}$ and with $c=c^\star_2$. 
%%We will denote this function by $\overline{u}_2$. In other words, $\overline{u}_2$ is a super-solution for the case
%%\[
%%r(t,x)=\tilde{r}(x-c_{A,2}t).
%%\]
It gives an upper bound for $\underline{c}$. 
In order to make this function to serve as a super-solution in our later argument, we present a slightly modified version of Lemma~\ref{lem:super0} as follows:

\begin{lem}\label{lem:super0-a}
Assume $c_A>2\sqrt{-\lambda_1}$ and that $A(t) = c_A (t-\tau) + B$ for $t\in[\tau, \tau+T)$ for some constants $B, \tau \geq 0$ and $T>0$. 
Let~$\overline{u}$ be the function defined by~\eqref{super_step2}, with~$c = c^\star$ from~\eqref{eq:c*formula}. Then for any constant $C\geq 1$, $C\overline{u}(t-\tau, x-B)$ is a super-solution of \eqref{eq:main} for $t \in [\tau, \tau+T)$ and $x \in \mathbb{R}$.
\end{lem}

\begin{proof}
By the change of variables $s=t-\tau$, $y=x-B$, the above function is converted to $C\overline{u}(s,y)$ and \eqref{eq:main} is converted to
\[
\partial_s u = \partial_{yy} u + \hat{f}(s,y,u)\quad (0<s<T,\,y\in\R),
\]
where $\hat{f}(s,y,u):=f(s+\tau,y+B,u)$. Since $\hat{f}$ satisfies the same assumption as \eqref{ass:KPP} and \eqref{ass:piecewise_constant_r} for $\tau\in[0,T)$, $C\overline{u}(s,y)$ is a super-solution of the above equation by Lemma~\ref{lem:super0}. Coming back to the original variables proves the lemma.
\end{proof}
%%%%%%%

Next, in the time interval $I_{2n}=[t_{2n},t_{2n+1})$, we will modify the sub-solution $\underline{u}$ in \eqref{subsolstep4}. More precisely, we will consider a function $\underline{u}_{1,\delta}$ that satisfies
\begin{equation}\label{eq:underlineu1delta}
\partial_t \underline{u}_{1,\delta} \leq \partial_{xx} \underline{u}_{1,\delta} +  (\tilde{r} (x-c_{A,1} t) - \delta) \underline{u}_{1,\delta} - M \underline{u}_{1,\delta}^2,
\end{equation}
and propagates with some speed $c  \geq   c^\star_1 - \varepsilon$. Notice that the spreading speed~$c^\star$ depends continuously on the values $r_1,r_2,r_3$, so that the above speed inequality holds true if $\delta$ is sufficiently small. This will give a lower bound for $\overline{c}$.

With these notations, we will proceed in two steps to prove Theorem~\ref{thm:intermediate_case_large_oscillations}. 

\begin{proof}[Step 1: 
Proof of\hspace{2pt} $\underline{c}\leq c^\star_2+2\varepsilon$]

Let $\overline{u}_2$ denote the function $\overline{u}$ in \eqref{super_step2} with $c=c^\star_2$ as in \eqref{c*_12}, which 
coincides with $c^\star$ in \eqref{eq:c*formula} for the case $2\sqrt{-\lambda_1}<c_A<2\sqrt{r_1}+2\sqrt{-\lambda_1-r_1}$, where $c_A$ is replaced by $c_{A,2}$. 
For each $n\in\N$, since $A(t)=A(t_{2n+1})+c_{A,2}(t-t_{2n+1})$ for $t\in[t_{2n+1},t_{2n+2})$, by Lemma~\ref{lem:super0-a}, $C\overline{u}_2(t-t_{2n+1},x-A(t_{2n+1})$ is a super-solution of \eqref{eq:main} for $t\in[t_{2n+1},t_{2n+2})$ and for any constant $C\geq 1$.
Recall also that $\overline{u}_2(x,t)$ decays exponentially as $x\to +\infty$, and that its front propagates at the speed $c=c^\star_2$. 

Let us also introduce an auxiliary super-solution $u^\#$. 
Let $\eta>0$ and let
\begin{equation}
    \lambda^\#=\frac{c_{A,2}}{2}+\sqrt{-r_3-\lambda_1}+\eta,\quad c^\#=\lambda^\#+\frac{r_2}{\lambda^\#}.
\end{equation}
Choose $K>1$ such that $K \min\left(1,  \upe^{-\lambda^\# x}\right)\geq u_0(x)$ for all $x\in\R$. 
By construction, $c^\#\geq 2\sqrt{r_2}$ and 
\begin{equation*}
(\lambda^\#)^2 - c^\# \lambda^\#  = -r_2 \leq -r(t,x),
\end{equation*}
for all $t >0$ and $x \in \mathbb{R}$. Thus
\begin{equation*}
u^\#:(t,x)\mapsto \min\left(1,\upe^{-\lambda^\#(x-c^\# t)}\right)
\end{equation*}
is a super-solution of~\eqref{eq:main} satisfying $u^\#\geq u$ globally in time and space.
Moreover, we can choose $\eta$ large enough so that $(c^\# -c_{A,2})t_1>L$, which implies
\begin{equation*}
(c^\# -c_{A,2})t_{2n+1}>L\quad \hbox{for all}\ n\in\N.
\end{equation*}
Thanks to this auxiliary super-solution, we have a rough control of the decay of $u(t,x)$ as $x\to+\infty$ for any $t>0$. 

Next, with $\overline{u}_2$ as defined above, we construct a sequence $(K_n)_{n\in\N}$ such that
\begin{equation}\label{Kn}
%%u(t_{2n+1},x)\leq 
u^\#(t_{2n+1},x)\leq K_n\overline{u}_2(0,x-A(t_{2n+1}))\quad\text{for all }x\in\R,\ \hbox{all large } n\in\N.
\end{equation}
In order to show that such $(K_n)_{n\in\N}$ exists, it suffices to show that
\[
\sup_{x\in\R}\,\frac{u^\#(t_{2n+1},x)}{\overline{u}_2(0,x-A(t_{2n+1}))} < +\infty\quad \hbox{for all large}\ n\in\N.
\]
Let us estimate the above quantity.

First, in the region $x\geq c^\# t_{2n+1}$, since $(c^\# -c_{A,2})t_{2n+1}>L$, we have
\[
x \geq c_{A,2} t_{2n+1} + L \geq A (t_{2n+1}) + L .
\]
Therefore, by the definition of $\overline{u}_2$ in \eqref{super_step2} for the case $x\geq c_A t$,
\begin{eqnarray*}
\overline{u}_2 (0, x- A (t_{2n+1}) ) & =& e^{- \frac{c_{A,2}}{2} (x- A (t_{2n+1}))} \varphi_1 \left( \frac{x- A (t_{2n+1})}{L} \right) \\
& \geq & C e^{- \frac{c_{A,2}}{2} (x- A (t_{2n+1}))} e^{-\sqrt{-r_3 - \lambda_1} (x- A (t_{2n+1}))}  ,
\end{eqnarray*}
for some $C>0$. Here we used the fact that either $r_1 < r_3$, or $r_1 \geq r_3$ and $L > \underline{L}$ (see Remark~\ref{rem:lambda1-r1}), so that~$\varphi_1$ is defined by either~\eqref{caseL_0>underlineL} or~\eqref{eq:eigen_case2}. Moreover, from our assumptions on the sequence $(t_n)_{n \in \N}$ we have that
$$A (t_{2n+1}) = c_{A,1} t_{2n+1} +  o (t_{2n+1}),$$
as $n \to +\infty$.
It follows that, by our choice of $\lambda^\#$ and $\eta >0$,
\begin{eqnarray*}
\frac{u^\#(t_{2n+1},x)}{\overline{u}_2(0,x-A(t_{2n+1}))} &\leq &  \frac{K}{C} \upe^{-\lambda^\# (x - c^\# t_{2n+1} )} e^{\left( \frac{c_{A,2}}{2} + \sqrt{-r_3 - \lambda_1} \right)(x- A (t_{2n+1}))} \\
&\leq &  \frac{K}{C}  e^{\left( \frac{c_{A,2}}{2} + \sqrt{-r_3 - \lambda_1}  \right)(c^\# t_{2n+1}- A (t_{2n+1}))} \\ 
& \leq & \frac{K}{C} \upe^{\left(\frac{c_{A,2}}{2}+\sqrt{-r_3-\lambda_1}\right)(c^\#-c_{A,1} +1)t_{2n+1}},
\end{eqnarray*}
for all $x \geq c^\# t_{2n+1}$.

Next, in the region $x\leq c^\# t_{2n+1}$, we have $u^\# (t_{2n+1},x) = K$, and, by \eqref{super_step2}, 
\[
\overline{u}_2 (0,x - A (t_{2n+1})) = \min \left( 2 , e^{-\lambda (c^\star_2) (x- A (t_{2n+1}))} \right) \geq 1,
\]
if $x \leq A (t_{2n+1})$, while using also the definition of $\varphi_1$ in either~\eqref{caseL_0>underlineL} or~\eqref{eq:eigen_case2},
$$
\overline{u}_2 (0, x - A (t_{2n+1})) \geq \upe^{-\frac{c_{A,2}(x - A(t_{2n+1}) )}{2}} \min_{[0,1]} \varphi_1 \times e^{-\sqrt{-r_3 - \lambda_1} (x - A (t_{2n+1}))}
$$
if $A (t_{2n+1})\leq x\leq c^\# t_{2n+1}$. Thus
$$\inf_{x \leq c^\# t_{2n+1} } \overline{u}_2 (0,x-A(t_{2n+1})) \geq \min_{[0,1]} \varphi_1  \times  e^{-\left( \frac{c_{A,2}}{2} + \sqrt{-r_3 - \lambda_1}\right) (c^\# t_{2n+1} - A (t_{2n+1}))} >0.$$ 
Hence
\begin{eqnarray*}
        \frac{u^\#(t_{2n+1},x)}{\overline{u}_2(0,x-A(t_{2n+1}))}& \leq &  \frac{K}{\min_{[0,1]} \varphi_1}  \times  e^{\left( \frac{c_{A,2}}{2} + \sqrt{-r_3 - \lambda_1}\right)(c^\# t_{2n+1} - A (t_{2n+1}))}\\
& \leq & \frac{K}{\min_{[0,1]} \varphi_1}  \times  e^{\left( \frac{c_{A,2}}{2} + \sqrt{-r_3 - \lambda_1}\right) (c^\#- c_{A,1} + 1) t_{2n+1} }\\
    \end{eqnarray*} 
for $x \leq c^\# t_{2n+1}$.

Therefore, for \eqref{Kn} to hold, it suffices to define the sequence $(K_n)$ by
\begin{equation}\label{eq:Kn}
    K_n= K  \max \left( \frac{1}{C}, \frac{1}{\min_{[0,1]} \varphi_1 } \right) \upe^{\left(\frac{c_{A,2}}{2}+\sqrt{-r_3-\lambda_1}\right)(c^\#-c_{A,1}+1)t_{2n+1}},
\end{equation}
for all large~$n$. Combining \eqref{Kn} and the inequality $u^\#\geq u$, we obtain
\[
u(t_{2n+1},x)\leq u^\#(t_{2n+1},x)\leq K_n \overline{u}_2(0,x-A(t_{2n+1})). 
\]
As mentioned before, $K_n \overline{u}_2(t-t_{2n+1},x-A(t_{2n+1}))$ is a super-solution in $t\in I_{2n+1}$, hence, by the comparison principle, we have for~$n$ large enough that
\[
u(t,x)\leq K_n\overline{u}_2(t-t_{2n+1},x-A(t_{2n+1}))\quad\text{for all}\ x\in\R,\ t\in I_{2n+1}.
\]
In particular, setting $t=t_{2n+2}$, we obtain
\begin{equation}\label{u<Kubar}
u(t_{2n+2},x)\leq K_n\overline{u}_2(|I_{2n+1}|,x-A(t_{2n+1}))\quad\text{for all}\ x\in\R.
\end{equation}
Let $B_n$ denote the value of the right-hand side of \eqref{u<Kubar} at $x=x_n: =(c^\star_2+2\varepsilon) t_{2n+2}$, that is,
\[
B_n:=K_n\overline{u}_2\left(|I_{2n+1}|,t_{2n+2}\left(c^\star_2+2\varepsilon-\frac{A(t_{2n+1})}{t_{2n+2}}\right)\right).
\]
Then, by \eqref{super_step2} we have that
\begin{equation}\label{Bn}
\begin{split}
B_n &= K_n e^{-\lambda (c^\star_2 )  \left( 2 \varepsilon t_{2n+2} - A (t_{2n+1}) + c^\star_2   t_{2n+1}  \right)}\\
& =K_n\upe^{-2\lambda(c^\star_2)\varepsilon t_{2n+2}}\upe^{\lambda(c^\star_2)\left(A(t_{2n+1})- c^\star_2 t_{2n+1}\right)},
\end{split}
\end{equation}
provided that
\begin{equation}
c^\star_2 (t_{2n+2}-t_{2n+1})-\frac{\ln 2}{\lambda(c^\star_2)}
\leq t_{2n+2}\left(c^\star_2+2\varepsilon-\frac{A(t_{2n+1})}{t_{2n+2}}\right)\leq c_{A,2}(t_{2n+2}-t_{2n+1}),
\end{equation}
which is clearly true for $n$ large enough and $\varepsilon >0$ small, since $t_{2n+1}=o(t_{2n+2})$ as~$n\to+\infty$.
Using~\eqref{eq:Kn} and again the assumption $t_{2n+1}=o(t_{2n+2})$, we see that the right-hand side of \eqref{Bn} tends to $0$ as $n\to+\infty$. It follows that
\[
\max_{x\geq x_n} K_n\overline{u}_2(|I_{2n+1}|,x-A(t_{2n+1}))\to 0\quad\hbox{as}\ \ n\to +\infty,
\]
since $\overline{u}_2(|I_{2n+1}|,x-A(t_{2n+1}))$ is monotone decreasing in $x$. This and \eqref{u<Kubar} imply
\[
\max_{x\geq x_n} u(t_{2n+2},x)\to 0\quad\hbox{as}\ \ n\to +\infty,
\]
which proves $\underline{c}\leq c^\star_2+2\varepsilon$.
\end{proof}

\begin{proof}[Step 2: Proof of\hspace{2pt} $\overline{c}\geq c^\star_1-2\varepsilon$]
By virtue of the global boundedness of the solution $u$ of \eqref{eq:main}, there exists $\rho > 0$ %$\rho \in\R$ (possibly negative) 
such that 
$f(t,x,u(t,x))\geq -\rho u(t,x)$ for any $(t,x)\in[0,+\infty)\times\R$. Therefore, by the comparison principle, for any $t>0$ and $x\in\R$,
\begin{equation}
    u(t,x)\geq\frac{e^{-\rho t}}{\sqrt{4\pi t}}\int_{\R}\upe^{-\frac{(x-y)^2}{4t}}u_0(y)\upd y.
\end{equation}
Then, there exists $\beta>0$ and $x_0\in\R$ such that $u_0\geq\beta \mathbf{1}_{[x_0-\beta,x_0+\beta]}$. 
It follows that, for any $c>0$ and $t>0$,
\begin{equation}
    u(t,ct)\geq\frac{\beta\upe^{-\rho t - \frac{c^2}{4} t}}{\sqrt{4\pi t}}\int_{x_0-\beta}^{x_0+\beta}\upe^{\frac{cy}{2}-\frac{y^2}{4t}}\upd y
    \geq\frac{\beta^2\upe^{-\rho t -\frac{c^2}{4} t}}{\sqrt{\pi t}}\upe^{\frac{c(x_0-\beta)}{2}}\upe^{-\frac{x_0^2+\beta^2+2|x_0|\beta}{4t}}.
\end{equation}
Recall that $\underline{u}_{1,\delta}$ is a subsolution of~\eqref{eq:underlineu1delta}, as constructed in~\eqref{subsolstep4} with $c \geq c^*_1 - \varepsilon$; see also Lemma~\ref{lem:sub0}. For any $n\in\N$, the support of $\underline{u}_{1,\delta} (0,x-A(t_{2n}))$ is exactly
\begin{equation*}
    A(t_2n)+\left[\frac{\ln S}{\eta}+x_0 -2R',L R\right].
\end{equation*}
Consequently, for any $n\in\N$,
the support of $\underline{u}_{1,\delta}(0,x-A(t_{2n}))$ is included in $\left[c_{A,1}t_{2n}+\frac{\ln S}{\eta}+x_0  -2R',c_{A,2}t_{2n}+L R\right]$.
Provided $n$ is sufficiently large, say $n\geq n_0$, the support is included in $[(c_{A,1}-\varepsilon)t_{2n},(c_{A,2}+\varepsilon)t_{2n}]$.
In such an interval, the decay of $x\mapsto u(t_{2n},x)$ can therefore be estimated
as follows (up to increasing~$n_0$):
\begin{equation}
    \min_{x\in[(c_{A,1}-\varepsilon)t_{2n},(c_{A,2}+\varepsilon)t_{2n}]}u(t_{2n},x)
    \geq C\upe^{-C' t_{2n}},
\end{equation}
where $C,C'>0$ are constants that only depend on $\beta$, $x_0$, $\rho$, $\delta$, $c_{A,1}$, $c_{A,2}$, $\varepsilon$.

Now, defining a sequence $(\kappa_n)_{n\in\N}$ by 
\begin{equation}\label{eq:kappan}
    \kappa_n = \frac{C\upe^{-C' t_{2n}}}{\max_{x\in\R}\underline{u}_{1,\delta}(0,x)}\quad\text{for all }n\in\N,
\end{equation}
we deduce
\begin{equation}\label{osc_sub_comp_t2n}
    \kappa_n\underline{u}_{1,\delta}(0,x-A(t_{2n}))\leq u(t_{2n},x)\quad\text{for all }n\geq n_0,\ x\in\R,
\end{equation}
and then
\begin{equation}
    \kappa_n\underline{u}_{1,\delta}(t,x-A(t_{2n}))\leq u(t+t_{2n},x)\quad\text{for all }n\geq n_0,\ x\in\R, \ t >0,
\end{equation}
by the parabolic comparison principle.

At this point of the proof, we remind that $\underline{u}_{1,\delta}$ travels with speed $c \geq c_1^\star - \varepsilon$ and is associated, not to the reaction term $f$, but to $(t,x,v)\mapsto (r- \delta) v - M v^2 $. More precisely, it satisfies~\eqref{eq:underlineu1delta}, \textit{i.e.},
$$
\partial_t \underline{u}_{1,\delta} \leq \partial_{xx} \underline{u}_{1,\delta} + (\tilde{r}(x-c_{A,1} t) - \delta) \underline{u}_{1,\delta} - M \underline{u}_{1,\delta}^2,
$$
where $\tilde{r}$ is defined in~\eqref{eq:r_abuse}. Therefore, if we multiply the function $\kappa_n\underline{u}_{1,\delta} (t,x-A(t_{2n}))$
by $\upe^{\frac{\delta}{2}t}$, we should still obtain a sub-solution, at least as long as this sub-solution is small enough.

Let us verify this last claim. First define 
\begin{equation*}
g:(t,x,v)\mapsto \tilde{r}(x-c_{A,1} t) v - M v^2,
\end{equation*}
so that
\begin{equation*}
g (t - t_{2n} ,x - A (t_{2n}),v) \leq f (t , x ,v),
\end{equation*}
for any $t \in [t_{2n}, t_{2n+1} ]$, $x \in \mathbb{R}$ and $v \geq 0$. 
Also we define
\begin{equation*}
\underline{u}_{1,\delta}^n:(t,x)\mapsto\underline{u}_{1,\delta} (t-t_{2n},x-A(t_{2n})).
\end{equation*}
Using~\eqref{eq:underlineu1delta}, the function
$(t,x)\mapsto \kappa_n\upe^{\frac{\delta}{2}(t-t_{2n})}\underline{u}_{1,\delta}^n(t,x)$
is a sub-solution of \eqref{eq:main} as long as $t \leq t_{2n+1}$ and
\begin{equation}
   M  \left( \kappa_n e^{\frac{\delta}{2} (t-t_{2n})} -1 \right) \underline{u}_{1,\delta}^n (t,x) \leq  \frac{\delta}{2} .
\end{equation}
Now recall that $\underline{u}_{1,\delta}$ is globally bounded by a constant $C''>0$ that depends on $\delta$ (hence on $\varepsilon$) but not on $n$. 
This constant is also an upper bound, independent of $n$, for each function $\underline{u}_{1,\delta}^n$.
Then the function
$(t,x)\mapsto \kappa_n\upe^{\frac{\delta}{2}(t-t_{2n})}\underline{u}_{1,\delta}^n(t,x)$
is a sub-solution of \eqref{eq:main} provided that
y
\begin{equation*}
    t-t_{2n}\leq T_n =\frac{2}{\delta}\ln\left(\frac{1}{\kappa_n}\left(1+\frac{\delta }{2 M C''}\right)\right) < t_{2n+1} - t_{2n},
\end{equation*}
where the latter inequality follows from the fact that $\frac{t_{2n+1}}{t_{2n}} \to +\infty$ and $| \ln \kappa_n | = O (t_{2n})$ by~\eqref{eq:kappan}.

By virtue of the comparison principle and~\eqref{osc_sub_comp_t2n},
\begin{equation}
    \kappa_n\upe^{\frac{\delta}{2}(t-t_{2n})}\underline{u}_{1,\delta}^n(t,x)\leq u(t,x)\quad\text{for all }n\geq n_0,\ x\in\R,\ t\in [t_{2n},t_{2n}+T_n].
\end{equation}
By a similar sub-solution construction, not growing in time but with an appropriately chosen amplitude (the algebra
is exactly the same), we can then prove that
\begin{equation}
    \kappa_n\upe^{\frac{\delta}{2}T_n}\underline{u}_{1,\delta}^n(t,x)\leq u(t,x)\quad\text{for all }n\geq n_0,\ x\in\R,\ t\in [t_{2n}+T_n,t_{2n+1}].
\end{equation}

Noting that $\kappa_n\upe^{\frac{\delta}{2}T_n} = 1+\frac{\delta }{2 M C''}$, we deduce
\begin{equation}
    \left( 1+\frac{\delta }{2 M C''}\right)\underline{u}_{1,\delta}^n(t_{2n+1},(c^\star_1-2\varepsilon)t_{2n+1})\leq u(t_{2n+1},(c^\star_1-2\varepsilon)t_{2n+1})\quad\text{for all }n\geq n_0.
\end{equation}
Using $t_{2n}=o(t_{2n+1})$ together with the detailed formula~\eqref{subsolstep4} where $c \geq c^\star_1 - \varepsilon$, it follows that
\begin{equation*}
    \underline{u}_{1,\delta}^n(t_{2n+1},(c^\star_1-2\varepsilon)t_{2n+1})= \underline{u}_{1,\delta} (t_{2n+1} - t_{2n}, (c^\star_1 - 2 \varepsilon) t_{2n+1} - A (t_{2n}) ) = \iota\sigma/2>0.
\end{equation*}
Finally we conclude that $\overline{c} \geq c^\star_1 - 2 \varepsilon$. This ends the proof of Step 2. As $\varepsilon$ can be chosen arbitrarily small, the proof of Theorem~\ref{thm:intermediate_case_large_oscillations} is complete.
\end{proof}

%%%%%%%%%%%%%%%%%%%%%
%%%%%%%%%%%%%%%%%%%%%
\section{Properties of $\lambda_1$}\label{sec:properties_lambda1}

Let $\lambda_1$ denote, as before, the principal eigenvalue of \eqref{eq:eigenproblem1} characterized by Proposition \ref{prop:eigen}. 
In this section, we uncover new properties of the map $(L,r_1,r_2,r_3)\mapsto\lambda_1$. Recall that if
\begin{equation}
    L>\underline{L}=
    \begin{cases}
    0 & \text{if }r_1=r_3, \\
    \frac{1}{\sqrt{r_2-\max(r_1,r_3)}}\operatorname{arccot}\left(\sqrt{\frac{r_2-\max(r_1,r_3)}{|r_1-r_3|}}\right) & \text{if }r_1\neq r_3,
    \end{cases}
\end{equation}
then $\lambda_1$ is characterized as the unique solution
in $\left(-r_2,\min\left(-\max\left(r_1,r_3\right),\frac{\pi^2}{L^2}-r_2\right)\right)$ of the equation \eqref{eq:equation_defining_lambda1},
whereas $\lambda_1 = - \max (r_1,r_3)$ if $L \leq \underline{L}$.
Recall also that $(L,r_1,r_2,r_3)\mapsto\lambda_1$ 
is smooth in each of the parameter sets $\{L>\underline{L}\}$, $\{L\leq\underline{L}, r_1<r_3\}$ and 
$\{L\leq\underline{L}, r_1>r_3\}$.

%%%%%%%%%%%%%%%%%%
\subsection{Monotonicity and symmetry}

According to \cite[Proposition 2.3, (vii)]{Berestycki_Ros}, the function $(L,r_1,r_2,r_3)\mapsto\lambda_1$ is nonincreasing and
concave with respect to each variable. It is also $1$-Lipschitz-continuous with respect to $r_1$, $r_2$ and $r_3$. 
The following two propositions state the monotonicity properties of $\lambda_1$ with respect to $L$ and $r_2$ in a more precise manner. 

\begin{prop}\label{prop:lambda1_L0}
The map $L\in(0,+\infty)\mapsto\lambda_1(L)$ is continuous, constant in $(0,\underline{L}]$ (this interval might be empty), 
decreasing in $(\underline{L},+\infty)$, with the following asymptotic or particular values:
\begin{itemize}
\item $\lambda_1(L)\to-\max(r_1,r_3)$ as $L\to\underline{L}^+$;
\item $\lambda_1\left(\frac{\pi}{2}\sqrt{\frac{2r_2-r_1-r_3}{(r_2-r_1)(r_2-r_3)}}\right)=-\frac{r_2^2-r_1r_3}{2r_2-r_1-r_3}$;
\item $\lambda_1(L)\to-r_2$ as $L\to+\infty$.
\end{itemize}
\end{prop}

\begin{proof}
The monotonicity of $L\in(\underline{L},+\infty)\mapsto\lambda_1(L)$ follows directly from the chain rule applied 
to the equation satisfied by $\lambda_1(L)$; in fact, $\frac{\partial\lambda_1}{\partial L}<0$.

The limit as $L\to+\infty$ is an immediate consequence of $L\sqrt{r_2+\lambda_1}<\pi$. The limit as $L\to\underline{L}^+$ follows similarly, separating the case $r_1=r_3$ and the case $r_1\neq r_3$. 
The continuity at $L=\underline{L}$ when $\underline{L}>0$ follows.

The function 
\begin{equation}
\lambda\mapsto\frac{r_2+\lambda-\sqrt{(r_1+\lambda)(r_3+\lambda)}}{\sqrt{r_2+\lambda}(\sqrt{-r_1-\lambda}+\sqrt{-r_3-\lambda})}
\end{equation}
is increasing, continuous and maps $(-r_2,-\max(r_1,r_3))$ onto $(-\infty,+\infty)$ if $r_1=r_3$ or
onto $\left(-\infty, \sqrt{\frac{r_2-\max(r_1,r_3)}{|r_1-r_3|}}\right)$ if $r_1\neq r_3$. Denote in both
cases $\overline{\zeta} \in (0,+\infty) \cup \{ + \infty\}$ the upper limit of this image interval.
Consequently, for any $\zeta\in(-\infty,\overline{\zeta})$, there exists
a unique preimage $\lambda_1\in(-r_2,-\max(r_1,r_3))$ such that 
\begin{equation}
\frac{r_2+\lambda_1-\sqrt{(r_1+\lambda_1)(r_3+\lambda_1)}}{\sqrt{r_2+\lambda_1}(\sqrt{-r_1-\lambda_1}+\sqrt{-r_3-\lambda_1})}=\zeta,
\end{equation}
and subsequently there exists a unique $L_0\in(\underline{L},+\infty)$ such that 
$\lambda_1=\lambda_1(L_0)$, with~$L_0$ given by the following formula:
\begin{equation}
L_0=\frac{1}{\sqrt{r_2+\lambda_1}}\operatorname{arccot}(\zeta).
\end{equation}

Now we are in a position to pick admissible values of $\zeta$ that correspond to remarkable values of the
$\cot$ function. For instance, $\zeta=0$ leads to:
\begin{equation}
\begin{cases}
L_0\sqrt{r_2+\lambda_1(L_0)}=\frac{\pi}{2}, \\
r_2+\lambda_1(L_0)=\sqrt{(r_1+\lambda_1(L_0))(r_3+\lambda_1(L_0))}.
\end{cases}
\end{equation}
After elementary manipulations, we deduce:
\begin{equation}
\begin{cases}
\displaystyle\lambda_1(L_0)=-\frac{r_2^2-r_1r_3}{2r_2-r_1-r_3}, \\
\displaystyle L_0=\frac{\pi}{2}\sqrt{\frac{2r_2-r_1-r_3}{(r_2-r_1)(r_2-r_3)}},
\end{cases}
\end{equation}
which completes the proof.
\end{proof}

Next, the function $\underline{L}$ is either $0$ or a decreasing continuous function of $r_2$. In the latter case, it also satisfies that $\underline{L} \to 0$ as $r_2 \to +\infty$. Thus, given a fixed value of~$L$, there exists
a threshold $\underline{r_2}\geq\max(r_1,r_3)$ such that $L > \underline{L}$ for any $r_2>\underline{r_2}$, and $L < \underline{L}$ for any $r_2 \in (\max (r_1,r_3), \underline{r}_2)$ (this latter interval possibly being empty). 
Thus, in the latter case, $\lambda_1 = - \max (r_1,r_3)$ as mentioned above. 
Quite similarly to the previous proposition, we deduce the following result; for the sake of brevity, we omit the proof.

\begin{prop}\label{prop:lambda1_r2}
The map $r_2\in(\max(r_1,r_3),+\infty)\mapsto\lambda_1(r_2)$ is continuous, constant in $(\max(r_1,r_3),\underline{r_2}]$
(this interval might be empty), decreasing in $(\underline{r_2},+\infty)$, with the following asymptotic or particular values:
\begin{itemize}
\item $\lambda_1(r_2)\to-\max(r_1,r_3)$ as $r_2\to\max (r_1,r_3)$; 
\item $\lambda_1(r_2)\to-\infty$ as $r_2\to+\infty$.
\end{itemize}
\end{prop}

%\subsection{Symmetry}
Finally, we point out that the eigenvalue $\lambda_1$ is symmetric with respect to the parameters $r_1$ and $r_3$.
\begin{prop}[Symmetry] 
$\lambda_1(r_1,r_3)=\lambda_1(r_3,r_1)$.
%%The map $(r_1,r_3)\mapsto\lambda_1(r_1,r_3)$ is symmetric.
\end{prop}
\begin{proof}
By the change of the variable $y\mapsto 1-y$ in \eqref{eq:eigenproblem1}, the role of $r_1$ and $r_3$ are exchanged, therefore the above symmetry is obvious.
\end{proof}

%%%%%%%%%%%%%%
\subsection{More general heterogeneities and optimization issues}\label{sec:moregeneral}

In view of our method of proof, it should be clear that the main result (the piece-by-piece formula 
for the spreading speed in Theorem \ref{thm:constant_case}) will extend to many equations of the form 
\begin{equation}
    \partial_t u -\partial_{xx}u=f(t,x,u)
\end{equation}
with $\partial_u f(t,x,0)=m(x-c_At)$, $m\in L^\infty(\R)$, $\inf m>0$. More precisely, since
the existence of $\lambda_1$ and $\varphi_1$ is given by \cite{Berestycki_Ros}, 
the following properties are the only true requirements to end the construction of super- and sub-solutions in the proof
of Theorem \ref{thm:constant_case}:
\begin{enumerate}
    \item $m=r_3$ in a neighborhood of $+\infty$ -- this is used to construct the explicit super-solution in Step 1;
    \item $m=r_1$ in a neighborhood of $-\infty$ -- this is used to construct the explicit super-solution in Step 2;
    \item $\limsup \frac{\varphi_1'}{\varphi_1} \leq \sqrt{-r_1-\lambda_1}$ at $-\infty$ -- this is used to validate the angle condition in Step 2;
    \item $\upe^{-\frac{c_A}{2}x}\varphi_1(x)\to 0$ at $+\infty$ -- this is used to ensure that the super-solution of Step 2
    acts indeed as a barrier for the solution.
\end{enumerate}
It turns out that the third and fourth ones are direct consequences of the first and second. Indeed, recall from Proposition~\ref{prop:eigen} (see also again~\cite{Berestycki_Ros}) that $\varphi_1$ is the limit of the eigenfunctions of a truncated Dirichlet problem. Using this together with the maximum principle and $\lambda_1 \leq -\max (r_1,r_3)$, one may check that $\varphi_1$ is either affine or identical to
$$e^{\sqrt{-r_1 - \lambda_1}x}$$
on a left-half line depending on whether $\lambda_1 = -r_1$ or $\lambda_1 < -r_1$. In the former case, $\varphi_1 '$ must be nonnegative and in the latter, $\varphi_1 '$ is precisely equal to $\sqrt{-r_1 -\lambda_1} \varphi_1$, that is, (3) holds true. The same argument shows that $\varphi_1$ grows at most linearly at~$+\infty$. 

Therefore the exact variations of $m$ between the left half-line where $m=r_1$ and the right half-line where $m=r_3$
do not really matter, and we are able to handle arbitrary bounded variations\footnote{We actually conjecture
that $m\to r_1$ at $-\infty$ and $m\to r_3$ at $+\infty$ are sufficient for our purposes. This, however, requires more
work.}.

Of course, such extensions are made at the expense of the formula \eqref{eq:equation_defining_lambda1} that characterizes $\lambda_1$. 
But they make it possible to study, for instance, smooth growth rates or piecewise-constant growth rates with more than
one traveling patch. By doing so, we obtain a nice connection with classical shape optimization results. For instance,
if $r_1=r_3$, $m\geq r_1$, and $L^\infty$ and $L^1$ constraints are imposed on $m-r_1$, then the function~$m$ which minimizes the eigenvalue $\lambda_1$ (and in turn maximizes the spreading speed) is bang-bang and contains precisely only one patch (\textit{i.e.}, $m-r_1$ is the indicator function of a bounded interval). In other words, the situation we studied in the present paper corresponds to this optimal situation.

%%%%%%%%%%%%%%%
\subsection{Comparison with related climate change models}

\revision{If we let $r_1,r_3\to 0$ in our model, the only favorable zone is the interval $[c_A t,c_A t +L]$, and there is no growth outside this moving interval. 
This situation has much similarity with the so-called climate change models. In this subsection, we make a brief comparison with these models.}

\revision{In the so-called climate change model studied in \cite{Berestycki_Die}, the growth rates outside the favorite zone are negative, which corresponds to $r_1,r_3<0$ in our notation. If we let $r_1,r_3\to 0$ in the climate change model, we formally obtain the same limit equation as our model. Let us compare these two limit models.}
%%We will also consider the limit as $r_1,r_3\to -\infty$ and compare it with the model that imposes homogeneous Dirichlet conditions on the boundary of the interval $[c_A t,c_A t +L]$.}

\revision{In our model, the limit $r_1=r_3\to 0$ with fixed $r_2>0$ and $L>0$ gives formally the following formula for the spreading speed:
\begin{equation*}
\underline{c}=\overline{c}=
\begin{cases}
c_A & \text{if }0\leq c_A < 2\sqrt{-\lambda_1}, \\
0 & \text{if } 2\sqrt{-\lambda_1} < c_A,
\end{cases}
\end{equation*}
where we have excluded the critical case~$2\sqrt{-\lambda_1} = c_A$ because of the discontinuity, and $\lambda_1\in(-r_2,0)$ is formally the solution in 
$\left(-r_2,\min\left(0,\frac{\pi^2}{L^2}-r_2\right)\right)$ of
\begin{equation*}
\cot(L\sqrt{r_2+\lambda_1})=\frac{r_2+2\lambda_1}{2\sqrt{r_2+\lambda_1}\sqrt{-\lambda_1}}.
\end{equation*}
Noticing the scale invariance and setting $r_2^L=L^2 r_2$ and $\lambda = \frac{\lambda_1}{r_2}$, 
we are seeking the solution $\lambda$ in $\left(-1,\min\left(0,\frac{\pi^2}{r_2^L}-1\right)\right)$ of 
\begin{equation*}
    \begin{split}
	\cot(\sqrt{r_2^L(1+\lambda)}) & = \frac{1+2\lambda}{2\sqrt{1+\lambda}\sqrt{-\lambda}} \\
	& = \frac{\sqrt{1+\lambda}^2-\sqrt{-\lambda}^2}{2\sqrt{-\lambda}\sqrt{1+\lambda}} \\
	& = \frac{\sqrt{1+\lambda}}{2\sqrt{-\lambda}}-\frac{\sqrt{-\lambda}}{2\sqrt{1+\lambda}} \\
	& = \frac12\sqrt{\frac{1+\lambda}{-\lambda}}-\frac12\sqrt{\frac{-\lambda}{1+\lambda}}.
    \end{split}
\end{equation*}
It appears that the right-hand side remains an increasing function of $\lambda$, whence the solution $\lambda_1$ is indeed
uniquely defined. It is negative, whence the threshold $2\sqrt{-\lambda_1}$ is positive, and the interval $[0,2\sqrt{-\lambda_1})$
is nonempty.}

\revision{Thus, formally, the solution should take the form of a spreading front with positive speed if and only if $c_A\in(0,2\sqrt{-\lambda_1})$, and in such a case $\underline{c}=\overline{c}=c_A$: the spreading front is locked. 
In particular, when $c_A$ is fixed, there exists a minimal value $\lambda_1=-\frac{c_A^2}{4}$ for 
the locked front. In terms of $L$, this means that the minimal critical length of $L$ for the front to be locked is given by 
\begin{equation*}
    L_{\textup{crit}}=\frac{1}{\sqrt{r_2-\frac{c_A^2}{4}}}\operatorname{arccot}\left( \frac{r_2-\frac{c_A^2}{2}}{c_A\sqrt{r_2-\frac{c_A^2}{4}}} \right).
\end{equation*}
}

\revision{This is consistent with the conclusions of \cite{Berestycki_Die}, which deals with the case $r_1=r_3<0$. There the critical value of $L$ for the persistence and spreading of a front at locked speed is given by \cite[Formula (23)]{Berestycki_Die}. Formally, the limit of this critical length, as $r_1=r_3\to 0$, is precisely the value $L_{\textup{crit}}$ above.
Thus, despite the fact that the steady state $0$ changes stability as $r_1=r_3$ changes sign, there is some continuity between the two regimes $r_1=r_3\to 0^+$ and $r_1=r_3\to 0^-$.
}

\revision{Next, if we consider the equation \eqref{eq:main} on the moving domain $[c_A t,c_A t+L]$ under the homogeneous Dirichlet boundary conditions, we are \secondrevision{led} to the following principal eigenvalue problem for the population growth:\secondrevision{
\begin{equation*}
    \begin{cases}
	-\varphi_1''-\left(r_2-\frac{c_A^2}{4}\right)\varphi_1=\lambda_1\varphi_1 & \text{in }(0,L), \\
	\varphi_1>0 & \text{in }(0,L), \\
	\varphi_1(0)=\varphi_1(L)=0. &
    \end{cases}
\end{equation*}
}
This ends up giving another (larger) critical length \secondrevision{$\widetilde{L}_{\textup{crit}}=\pi/\sqrt{r_2-\frac{c_A^2}{4}}$} for the persistence and spreading of a front at the speed $c_A$.}

\revision{This critical length coincides with the limit $r_1=r_3\to-\infty$ of the critical length in the 
model from \cite{Berestycki_Die} (see \cite[Formula (26)]{Berestycki_Die}). 
The fact that $\widetilde{L}_{crit}>L_{crit}$ illustrates that our model, even in the limit $r_1=r_3\to 0$, is much more favorable for population growth than the Dirichlet bounded domain model.} \secondrevision{We also refer to \cite{Allwright_2021}, where this problem is studied in a more general framework from which the above formula for $\widetilde{L}_{\textup{crit}}$ is recovered as a special case.}

\revision{Our model could also be compared with a Neumann bounded domain model. In this last case,
the associated principal eigenvalue problem is:
\begin{equation*}
    \begin{cases}
	-\varphi_1''-c_A\varphi_1'-r_2\varphi_1=\lambda_1\varphi_1 & \text{in }(0,L), \\
	\varphi_1>0 & \text{in }(0,L), \\
	\varphi_1'(0)=\varphi_1'(L)=0. &
    \end{cases}
\end{equation*}
Of course, here, $\lambda_1=-r_2$ does not depend on $c_A$ or $L$. There is no critical length, that is, population growth (hence the persistence of the front) occurs regardless of the value of $L$. 
This model is much more favorable for population growth than all other models considered above.}

%%%%%%%%%%%%%%%%%
%%%%%%%%%%%%%%%%%
\section*{Acknowledgment}

This work was initiated during a visit of T.G. and L.G. at the Meiji Institute for Advanced Studies of Mathematical Sciences (MIMS).
H.M. is supported by KAKENHI 21H00995. T.G. and L.G. acknowledge support from the ANR via the project Indyana under grant agreement ANR-21-CE40-0008 and via the project Reach under grant agreement ANR-23-CE40-0023-01. The three authors also acknowledge support from the CNRS via IRN ReaDiNet.

\revision{The authors thank the anonymous referee for many valuable comments which led to a substantial improvement of the original manuscript.}

%\typeout{}
\bibliographystyle{plain}
\bibliography{ref.bib}

\end{document}